\documentclass[11 pt]{amsart}

\usepackage{amsfonts}
\usepackage{amssymb}
\usepackage{amsmath}
\usepackage{latexsym}
\usepackage{mathrsfs}
\usepackage{amsthm}
\usepackage{verbatim}

\newtheorem{thrm}{Theorem}[section]
\newtheorem{lem}[thrm]{Lemma}
\newtheorem{cor}[thrm]{Corollary}
\newtheorem{prop}[thrm]{Proposition}

\theoremstyle{definition}
\newtheorem{defn}[thrm]{Definition}
\newtheorem{exmple}[thrm]{Example}
\newtheorem{rmk}[thrm]{Remark}

\newtheorem{ques}[thrm]{Question}

\begin{document}

\newcommand{\vol}{\mathrm{vol}}

\newcommand{\Supp}{\mathrm{Supp}}

\newcommand{\Sing}{\mathrm{Sing}}

\title{Comparing Numerical Dimensions}
\author{Brian Lehmann}
\thanks{This material is based upon work supported under a National Science
Foundation Graduate Research Fellowship.}
\address{Department of Mathematics, Rice University \\
Houston, TX \, \, 77005}
\email{blehmann@rice.edu}

\begin{abstract}
The numerical dimension is a numerical measure of the positivity of a pseudo-effective divisor $L$.  There are several proposed definitions of the numerical dimension due to \cite{nakayama04} and \cite{bdpp04}.  We prove the equality of these notions and give several additional characterizations.  We also prove some new properties of the numerical dimension.
\end{abstract}

\maketitle

\section{Introduction}

Suppose that $X$ is a smooth complex projective variety and $L$ is an effective divisor.  An important principle in birational geometry is that the geometry of $L$ is captured by the asymptotic behavior of the spaces $H^{0}(X,\mathcal{O}_{X}(mL))$ as $m$ increases.  When $L$ is a big divisor, this asymptotic behavior has close ties to the cohomological and numerical properties of $L$.  These connections have been applied profitably in many situations in birational geometry, most notably in the minimal model program.

However, when $L$ is an effective divisor that is not big, these close relationships no longer hold.  In order to understand the interplay between numerical and asymptotic properties, \cite{kawamata85} defined the numerical dimension of a nef divisor.  \cite{nakayama04} and \cite{bdpp04} proposed several different extensions of this notion to pseudo-effective divisors.  Our goal is to give a unifying framework for the numerical dimension by proving the equality of these definitions and giving  other natural descriptions as well.  We also describe some new properties of the numerical dimension.  The crucial perspectives are that:
\begin{enumerate}
\item The numerical dimension measures the asymptotic behavior of $L$ when it is perturbed by adding a small ample divisor $\epsilon A$.
\item The numerical dimension measures the largest dimension of a subvariety $W \subset X$ such that $L$ is positive along $W$.   An important subtlety is that one should not simply consider $L|_{W}$ but should ``remove'' contributions of the base locus of $L$.
\end{enumerate}

Since some of the definitions used in the main theorem are rather technical, we simply give references here.  We will describe the intuition behind the theorem in Section \ref{intuitivedesc}.  The notation $\mathbf{B}_{-}(L)$ denotes the diminished base locus defined in Section \ref{baseloci}, $\vol_{X|W}$ denotes the restricted volume defined in Section \ref{restrictedvolumesection},  $P_{\sigma}(-)$ denotes the divisorial Zariski decomposition defined in Section \ref{divzardecomsection}, and $\langle - \rangle$ denotes the restricted positive product defined in Section \ref{restrictedproductsection}.

\begin{thrm} \label{equalityofnumericaldimension}
Let $X$ be a normal projective variety over $\mathbb{C}$ and let $L$ be a pseudo-effective $\mathbb{R}$-Cartier $\mathbb{R}$-Weil divisor.  In the following,  $A$ will denote some fixed sufficiently ample $\mathbb{Z}$-divisor and $W$ will range over all subvarieties of $X$ not contained in $\mathbf{B}_{-}(L) \cup \Supp(L) \cup \Sing(X)$.  The following quantities coincide:

Perturbed growth condition:
\begin{enumerate}
\item $\max \left\{ k \in \mathbb{Z}_{\geq 0} \left| \limsup_{m \to \infty} \frac{h^{0}(X,\mathcal{O}_{X}(\lfloor mL \rfloor + A))}
{m^{k}} > 0 \right. \right\}$.
\end{enumerate}

Volume conditions:
\begin{enumerate}
\setcounter{enumi}{1}
\item $\max \{ k \in \mathbb{Z}_{\geq 0} | \,  \exists C>0 \textrm{ such that } Ct^{n-k} < \vol(L+tA) \textrm{ for all } t > 0 \}$.
\item $\max \{ \dim W | \lim_{\epsilon \to 0} \vol_{X|W}(L + \epsilon A) > 0 \}$.
\item $\max \{ \dim W | \inf_{\phi: Y \to X} \vol_{\widetilde{W}}(P_{\sigma}(\phi^{*}L)|_{\widetilde{W}}) >0 \}$ where $\phi$ varies over all birational maps such that no exceptional center contains $W$ and $\widetilde{W}$ denotes the strict transform of $W$.
\end{enumerate}

Positive product conditions:
\begin{enumerate}
\setcounter{enumi}{4}
\item $\max \{ k \in \mathbb{Z}_{\geq 0} | \langle L^{k} \rangle \neq 0 \}$.
\item $\max \{ \dim W | \langle L^{\dim W} \rangle_{X|W} > 0 \}$.
\end{enumerate}

Seshadri-type condition:
\begin{enumerate}
\setcounter{enumi}{6}
\item $\min \left\{ \dim W \left| \begin{array}{l} \phi^{*}L - \epsilon E \textrm{ is not
    pseudo-effective} \\ \textrm{for any } \epsilon > 0 \textrm{ where } \phi: Bl_{W}(X) \to X \\ \textrm{and } \mathcal{O}_{X}(-E) = \phi^{-1}\mathcal{I}_{W} \cdot \mathcal{O}_{Bl_{W}X} \end{array} \right. \right\}$.

    \noindent By convention, if $L$ is big we interpret this expression as returning $\dim(X)$.
\end{enumerate}
This common quantity is known as the numerical dimension of $L$, and is denoted $\nu(L)$.
It only depends on the numerical class of $L$.
\end{thrm}

The definitions $\kappa_{\sigma}$ and $\kappa_{\nu}$ of \cite{nakayama04} are listed as (1) and (7) respectively and the definition $\nu$ of \cite{bdpp04} is listed as (5).  When $L$ is nef this definition agrees with the definition of \cite{kawamata85}.

\begin{rmk}
The numerical dimension also admits a natural interpretation with respect to separation of jets, reduced volumes, and the other invariants considered in \cite{elmnp06}.
\end{rmk}

The numerical dimension is very natural from the viewpoint of birational geometry.  It is established in \cite{nakayama04} that for a pseudo-effective divisor $L$:
\begin{itemize}
\item $0 \leq \nu(L) \leq \dim(X)$.
\item $\nu(L) = \dim(X)$ iff $L$ is big and $\nu(L) = 0$ iff $P_{\sigma}(L) \equiv 0$.
\item $\kappa(L) \leq \nu(L)$.
\item If $\phi: Y \to X$ is a surjective morphism then $\nu(\phi^{*}L) = \nu(L)$. 
\end{itemize}
We will prove two additional basic properties, answering a question of \cite{nakayama04}:
\begin{itemize}
\item $\nu(L) = \nu(P_{\sigma}(L))$.
\item Fix some sufficiently ample $\mathbb{Z}$-divisor $A$.  Then there are positive constants $C_{1}, C_{2}$ such that
\begin{equation*}
C_{1}m^{\nu(L)} < h^{0}(X,\mathcal{O}_{X}(\lfloor mL \rfloor + A)) < C_{2}m^{\nu(L)}
\end{equation*}
for every sufficiently large $m$.
\end{itemize}
The properties of $\nu(L)$ will be discussed in more depth in Section \ref{numericaldimensionsection}.

\subsection{Intuitive Description} \label{intuitivedesc}

We now turn to an intuitive description of several of the definitions in Theorem \ref{equalityofnumericaldimension}.  Classically one measures the positivity of a divisor using the rate of growth of sections of $H^{0}(X,\mathcal{O}_{X}(mL))$ as $m$ increases.  More precisely, the Iitaka dimension is defined as
\begin{equation*}
\kappa(L) = \max \left\{ k \in \mathbb{Z}_{\geq 0} \left| \limsup_{m \to \infty} \frac{h^{0}(X,\mathcal{O}_{X}(\lfloor mL \rfloor))} {m^{k}} > 0 \right. \right\}.
\end{equation*}
(If $H^{0}(X,\mathcal{O}_{X}(\lfloor mL \rfloor))=0$ for every $m$, we set $\kappa(L) = -\infty$.)  To obtain a numerical invariant, we must instead consider sections of $mL + A$ for some sufficiently ample divisor $A$.  Thus definition (1) indicates that $\nu(L)$ can be viewed as a numerical analogue of the Iitaka dimension.

Another way to calculate the positivity of $L$ is to use intersection products.  \cite{kawamata85} defined the numerical dimension of a nef divisor $L$ as
\begin{equation*}
\nu(L) := \max \{ k \in \mathbb{Z}_{\geq 0} | L^{k} \cdot A^{n-k} \neq 0 \}
\end{equation*}
for some (thus any) ample divisor $A$.  The naive extension of this definition to pseudo-effective divisors does not work, as the diminished base locus of $L$ might contribute positively to this intersection and distort the measurement.  The positive product of \cite{bdpp04} gives a precise method of taking intersection products while discounting these contributions.  Definition (5) shows that $\nu(L)$ can be defined as in \cite{kawamata85} by replacing the intersection product by the positive product.

A third way to measure the positivity of a divisor is the volume: setting $n = \dim(X)$,
\begin{equation*}
\vol(L) := \limsup_{m \to \infty} \frac {h^{0}(X,\mathcal{O}_{X}(mL))} {m^{n}/n!}.
\end{equation*}
Conceptually, one can view the volume as a loose analogue of the top self-intersection of $L$.  While this latter quantity does not usually yield geometric information, the volume is a useful alternative that still shares many of the desirable properties of intersection products.  \cite{lm09} and \cite{bfj09} show that $\vol$ is a differentiable function on the space of big $\mathbb{R}$-Cartier divisors.  Definition (2) demonstrates that $\nu(L)$ controls the derivative of $\vol$ near $L$.

\subsection{Restricted Numerical Dimension}

It is useful to study not only numerical invariants on $X$ but also restricted versions that measure positivity along a subvariety $V$.  We will define a restricted numerical dimension of $L$ along a subvariety $V$ of $X$.  Just as in the non-restricted case, the restricted numerical dimension should measure the maximal dimension of a very general subvariety $W \subset V$ such that the ``positive restriction'' of $L$ is big along $W$.

\begin{defn} \label{firstrestnumdim}
Let $X$ be a smooth variety, $V$ a subvariety, and $L$ a pseudo-effective $\mathbb{R}$-divisor such that $V \not \subset \mathbf{B}_{-}(L)$.  Fix an ample divisor $A$.  We define the restricted numerical dimension $\nu_{X|V}(L)$ to be
\begin{equation*}
\nu_{X|V}(L) := \max \{ \dim W | \lim_{\epsilon \to 0} \vol_{X|W}(L + \epsilon A) > 0 \}
\end{equation*}
where $W$ ranges over smooth subvarieties of $V$ not contained in $\mathbf{B}_{-}(L)$.  The restricted numerical dimension is an invariant of the numerical class of $L$.
\end{defn}

The restricted numerical dimension satisfies (slightly weaker) analogues of Theorems \ref{equalityofnumericaldimension} and \ref{numdimproperties}.  For nef divisors we obtain nothing new because $\nu_{L|V}(L) = \nu_{V}(L|_{V})$.  Nevertheless, the restricted numerical dimension plays an important role in understanding the geometry of a pseudo-effective divisor $L$.

\subsection{Organization} The paper is organized as follows.  Section \ref{divzardecomsection} is devoted to the study of the divisorial Zariski decomposition, giving the technical background for the rest of the paper.  Sections \ref{restrictedproductsection} and \ref{nakayamaconstantsection} prove some basic facts about the invariants of Theorem \ref{equalityofnumericaldimension}.  We then turn to the proof of Theorem \ref{equalityofnumericaldimension} in Section \ref{numericaldimensionsection}.  Section \ref{restrictednumericaldimensionsection} is devoted to a discussion of the restricted numerical dimension.

I would like to thank my advisor J.~M\textsuperscript{c}Kernan for his advice and support.  I also thank R.~Lazarsfeld for some helpful conversations, Y.~Gongyo for pointing out several mistakes in an earlier draft, and the referee for the careful revisions.

\section{Preliminaries} \label{preliminaries}

All schemes will lie over the base field $\mathbb{C}$.  A variety will always be an irreducible reduced projective scheme.  The ambient variety $X$ is assumed to be normal unless otherwise noted.  The term ``divisor'' will always refer to an $\mathbb{R}$-Cartier $\mathbb{R}$-Weil divisor.  $N^{p}(X)$ will denote the $\mathbb{R}$-vector space of codimension $p$ cycles quotiented out by those numerically equivalent to $0$, and $CD(X)$ will denote the $\mathbb{R}$-vector space of Cartier divisors quotiented out by those that have degree $0$ along every irreducible curve.

\subsection{Base Loci} \label{baseloci}

Let $L$ be a pseudo-effective divisor.  The $\mathbb{R}$-stable base locus of $L$ is defined to be
\begin{equation*}
\mathbf{B}_{\mathbb{R}}(L) := \bigcap \{ \; \Supp(D) \; | \; D \geq 0 \; \; \textrm{and} \; \; D \sim_{\mathbb{R}} L \}.
\end{equation*}
When $L$ is not $\mathbb{R}$-linearly equivalent to an effective divisor, we use the convention that $\mathbf{B}_{\mathbb{R}}(L) = X$.  The $\mathbb{R}$-stable base locus is always a Zariski-closed subset of $X$; we do not associate any scheme structure to it.

We obtain a much better behaved invariant by perturbing by an ample divisor.  This approach to invariants was first considered in \cite{nakamaye00} and was studied systematically in \cite{elmnp05}.

\begin{defn}
Let $L$ be a pseudo-effective divisor.  The augmented base locus of $L$ is
\begin{equation*}
\mathbf{B}_{+}(L) := \bigcap_{A \textrm{ ample}} \mathbf{B}_{\mathbb{R}}(L - A).
\end{equation*}
\end{defn}

Note that $\mathbf{B}_{+}(L) \supset \mathbf{B}_{\mathbb{R}}(L)$.  \cite{elmnp05} Corollary 1.6 verifies that the augmented base locus is equal to $\mathbf{B}_{\mathbb{R}}(L - A)$ for any sufficiently small ample divisor $A$.  Thus $\mathbf{B}_{+}(L)$ is a Zariski-closed subset of $X$ and it only depends on the numerical class of $L$.

For the second variant, we add on a small ample divisor.

\begin{defn}
Let $L$ be a pseudo-effective divisor.  The diminished base locus of $L$ is
\begin{equation*}
\mathbf{B}_{-}(L) = \bigcup_{A \textrm{ ample}} \mathbf{B}_{\mathbb{R}}(L + A).
\end{equation*}
\end{defn}

\begin{rmk}
Although \cite{nakayama04} uses a different definition, it is equivalent to ours by \cite{nakayama04} Theorem V.1.3.
\end{rmk}

\cite{elmnp05} Proposition 1.15 checks that the diminished base locus only depends on the numerical class of $L$.  Unlike the augmented base locus, the diminished base locus is probably not a Zariski-closed subset (although no examples are known of such pathological behavior).  However, it is a countable union of closed subsets by the following theorem.

\begin{thrm}[\cite{nakayama04}, Theorem V.1.3] \label{restrictedbaselocuscountable}
Let $X$ be a smooth variety and let $L$ be a pseudo-effective divisor.  There is an ample divisor $A$ such that
\begin{equation*}
\mathbf{B}_{-}(L) = \bigcup_{m} \mathrm{Bs}(\lceil mL \rceil + A)
\end{equation*}
where $\mathrm{Bs}$ denotes the (set-theoretic) base locus.
\end{thrm}

\cite{nakayama04} proves the invariance of $\mathbf{B}_{-}(L)$ under surjective morphisms.

\begin{prop} \label{restrictedbaselocusinvariance}
Let $\phi: Y \to X$ be a surjective morphism from a normal variety $Y$ onto a normal variety $X$.  Suppose that $L$ is a pseudo-effective divisor on $X$.  Then we have an equality of sets
\begin{equation*}
\phi^{-1}\mathbf{B}_{-}(L) \cup \phi^{-1}\Sing(X) = \mathbf{B}_{-}(\phi^{*}L) \cup \phi^{-1}\Sing(X).
\end{equation*}
\end{prop}

\begin{proof}
Fix an ample divisor $H$ on $Y$ and an ample divisor $A$ on $X$.  We have
\begin{align*}
\phi^{-1}\mathbf{B}_{-}(L) & = \phi^{-1}\left( \bigcup_{m} \mathbf{B}_{\mathbb{R}}\left(L + \frac{1}{m}A \right) \right) \textrm{ by \cite{elmnp05} Remark 1.20} \\
& = \bigcup_{m} \mathbf{B}_{\mathbb{R}}\left( \phi^{*} \left(L + \frac{1}{m}A \right) \right) \\
& \supset \bigcup_{m} \mathbf{B}_{\mathbb{R}}\left( \phi^{*} \left(L + \frac{1}{m}A \right) + \frac{1}{m}H \right) \\
& = \mathbf{B}_{-}(\phi^{*}L) \textrm{ by \cite{elmnp05} Remark 1.20}.
\end{align*}
This proves the inclusion $\supset$.  Furthermore, the same argument shows that it suffices to prove the reverse inclusion $\subset$ after replacing $Y$ by any higher birational model.

We next reduce to the case where $X$ and $Y$ are smooth.  Let $\psi: \widetilde{X} \to X$ denote a resolution that is an isomorphism away from $\Sing(X)$.  Suppose that the closed point $\widetilde{x} \not \in \mathbf{B}_{-}(\phi^{*}L) \cup \phi^{-1}\Sing(X)$.  Fix an ample divisor $\widetilde{A}$ on $\widetilde{X}$ and choose an ample divisor $A$ on $X$ so that $\phi^{*}A - \widetilde{A}$ is an effective divisor $E$.  Since $\widetilde{x}$ is not contained in the $\psi$-exceptional locus, we may also ensure that that $\widetilde{x} \not \in \Supp(E)$.  Then
\begin{equation*}
\widetilde{x} \not \in \mathbf{B}_{\mathbb{R}}(\phi^{*}(L) + \epsilon H + \epsilon E) = \phi^{-1}\mathbf{B}_{\mathbb{R}}(\phi^{*}(L+ \epsilon A))
\end{equation*}
for any $\epsilon >0$, showing that
\begin{equation*}
\psi^{-1}\mathbf{B}_{-}(L) \cup \psi^{-1}\Sing(X) = \mathbf{B}_{-}(\psi^{*}L) \cup \psi^{-1}\Sing(X).
\end{equation*}
As discussed earlier, we may verify the desired equality of sets by replacing $Y$ by a smooth birational model that dominates $\widetilde{X}$.  Thus we have reduced to case when both $X$ and $Y$ are smooth.

\cite{nakayama04} Lemma III.2.3 and Lemma III.5.15 together show that for a smooth variety $Z$ and a  pseudo-effective divisor $M$ on $Z$, a closed point $z \in Z$ is contained in $\mathbf{B}_{-}(M)$ if and only if for every birational map $\psi: W \to Z$ from a smooth variety $W$ and every $\psi$-exceptional divisor $E$ with $\psi(E) = z$ we have $E \subset \mathbf{B}_{-}(\psi^{*}L)$.  This immediately implies the inclusion $\subset$ when both $X$ and $Y$ are smooth, concluding the proof.
\end{proof}

\subsection{$V$-pseudo-effective Cone and $V$-big Cone}

The perturbed base loci can be used to describe when a divisor $L$ sits in ``general position'' with respect to a subvariety $V$.

\begin{defn} \label{vbigcone}
Suppose that $V \subset X$ is a subvariety.  We define the $V$-pseudo-effective cone $\mathrm{Psef}_{V}(X)$ to be the cone in $CD(X)$ generated by classes of divisors $L$ with $V \not \subset \mathbf{B}_{-}(L)$.  We define the $V$-big cone $\mathrm{Big}_{V}(X)$ to be the cone generated by classes of divisors $L$ with $V \not \subset \mathbf{B}_{+}(L)$.
\end{defn}

It is easy to verify that $\mathrm{Psef}_{V}(X)$ is closed and $\mathrm{Big}_{V}(X)$ is its interior.  Note also that $L|_{V}$ is pseudo-effective whenever $L$ has numerical class in $\mathrm{Psef}_{V}(X)$.  The following perspective will sometimes be useful.

\begin{defn}
Suppose that $V \subset X$ is a subvariety.  If $L$ is an effective divisor such that $\Supp(L) \not \supset V$ we say $L \geq_{V} 0$.
\end{defn}

The relationship with the earlier criteria is given by a trivial lemma.

\begin{lem} \label{vbignesscondition}
Suppose that $V \subset X$ is a subvariety.  If $L$ is a $V$-big divisor, then $L \sim_{\mathbb{R}} L'$ for some $L' \geq_{V} 0$.
\end{lem}

\subsection{Admissible and V-Birational Models}

Suppose that $X$ is a normal variety and $V$ is a subvariety.  In order to study how $V$-pseudo-effective divisors behave under birational pull-backs, we need to be careful about how $V$ intersects the exceptional centers of the map.  The most general situation is the following.

\begin{defn}
Let $X$ be a normal variety and $V$ a subvariety of $X$.  Suppose that $\phi: Y \to X$ is a birational map and that $W$ is a subvariety of $Y$ such that the induced map $\phi|_{W}: W \to V$ is generically finite.  We say that $(Y,W)$ or $\phi: (Y,W) \to (X,V)$ is an admissible model for $(X,V)$.  When both $Y$ and $W$ are smooth, we say that $(Y,W)$ is a smooth admissible model.
\end{defn}

The disadvantage of admissible models is that in many circumstances we need to keep track of the degree of $\phi|_{W}$.  Since we want to focus on the birational geometry of $V$, we will usually restrict ourselves to the following situation.

\begin{defn}
Let $X$ be a normal variety and $V$ a subvariety not contained in $\Sing(X)$.  Suppose that $\phi: \widetilde{X} \to X$ is a birational map from a normal variety $\widetilde{X}$ such that $V$ is not contained in any $\phi$-exceptional center.  Let $\widetilde{V}$ denote the strict transform of $V$.  We say that $(\widetilde{X},\widetilde{V})$ or $\phi: \widetilde{X} \to X$ is a $V$-birational model for $(X,V)$.  When both $\widetilde{X}$ and $\widetilde{V}$ are smooth, we say that $(\widetilde{X},\widetilde{V})$ is a smooth $V$-birational model.
\end{defn}

Suppose that $V$ is a subvariety not contained in $\Sing(X)$ and $\phi: (Y,W) \to (X,V)$ is an admissible model.  By Proposition \ref{restrictedbaselocusinvariance}, the pull-back of a $V$-pseudo-effective divisor under $\phi$ is $W$-pseudo-effective.  If $\phi$ is a $V$-birational model, then more is true:

\begin{prop} \label{vbigconebirational}
Let $X$ be a normal variety and $V$ a subvariety not contained in $\Sing(X)$.  Suppose that $\phi: \widetilde{X} \to X$ is a $V$-birational model.  If $L$ is a $V$-big divisor then $\phi^{*}L$ is a $\widetilde{V}$-big divisor.
\end{prop}

\begin{proof}
$V$-pseudo-effectiveness of $L$ implies that $\phi^{*}L$ is $\widetilde{V}$-pseudo-effective.  By openness of the $\widetilde{V}$-big cone, it suffices to check that $\phi^{*}H$ is $\widetilde{V}$-big for an ample divisor $H$ on $X$.  Let $\psi: \widetilde{Y} \to \widetilde{X}$ be a smooth model such that $\psi$ is an isomorphism away from $\Sing(\widetilde{X})$.  Note that for some sufficiently small $\epsilon$,
\begin{align*}
\psi^{-1}\mathbf{B}_{+}(\phi^{*}H) & = \psi^{-1}\mathbf{B}_{\mathbb{R}}((1-\epsilon)\phi^{*}H) \textrm{ by \cite{elmnp05} Corollary 1.6}\\
& = \mathbf{B}_{\mathbb{R}}((1-\epsilon)\psi^{*}\phi^{*}H) \\
& \subset \mathbf{B}_{+}(\psi^{*}\phi^{*}H).
\end{align*}
But clearly $\mathbf{B}_{+}(\psi^{*}\phi^{*}H)$ is contained in the $(\phi \circ \psi)$-exceptional locus.  Thus $\mathbf{B}_{+}(\phi^{*}H)$ is contained inside the union of the $\phi$-exceptional locus and $\Sing(\widetilde{X})$.  In particular it does not contain $\widetilde{V}$.
\end{proof}

\subsection{Restricted Volume} \label{restrictedvolumesection}

Just as the volume measures the asymptotic rate of growth of sections, the restricted volume measures the rate of growth of restrictions of sections to a subvariety $V$.  This notion originated in the work of Hacon-M\textsuperscript{c}Kernan and Takayama and is systematically developed in \cite{elmnp06}.

\begin{defn} \label{restrictedvolumedefinition}
Suppose that $X$ is a normal variety, $V$ is a $d$-dimensional subvariety of $X$, and $L$ is a divisor.  We define
\begin{equation*}
H^{0}(X|V,\mathcal{O}_{X}(\lfloor L \rfloor)) := \mathrm{Im} (H^{0}(X,\mathcal{O}_{X}(\lfloor mL \rfloor)) \to H^{0}(V,\mathcal{O}_{V}(\lfloor mL \rfloor)))
\end{equation*}
and $h^{0}(X|V,\mathcal{O}_{X}(\lfloor L \rfloor))$ to be the dimension of this space.
We then define the restricted volume $\vol_{X|V}(L)$ to be
\begin{equation*}
\vol_{X|V}(L) := \limsup_{m \to \infty} \frac {h^{0}(X|V,\mathcal{O}_{X}(\lfloor mL \rfloor))} {m^{d}/d!}.
\end{equation*}
\end{defn}

\begin{rmk}
Although this definition of $\vol_{X|V}$ is formulated differently from that of \cite{elmnp06}, the two definitions agree (whenever the restricted volume is defined in \cite{elmnp06}).  An elementary argument proves that $\vol_{X|V}$ is homogeneous of degree $d$, so that Definition \ref{restrictedvolumedefinition} agrees with the definition in \cite{elmnp06} for $\mathbb{Q}$-divisors.  In particular, $\vol_{X|V}$ is a continuous function on the space of $V$-big $\mathbb{Q}$-divisors.  Using this fact, one readily checks that $\vol_{X|V}$ is continuous on the set of $V$-big $\mathbb{R}$-divisors by perturbing by ample divisors, and thus coincides with the definition of \cite{elmnp06}.
\end{rmk}

As with the other quantities we consider, the restricted volume is a numerical and birational invariant.  More precisely, \cite{elmnp06} Theorem A shows that if $L$ and $L'$ are numerically equivalent $V$-big divisors then $\vol_{X|V}(L) = \vol_{X|V}(L')$.  Furthermore, \cite{elmnp06} Proposition 2.4 proves that the restricted volume remains unchanged upon pulling back to an admissible model.

\subsection{Twisted Linear Series}

It was observed by Iitaka that linear series of the form $|\lfloor mL \rfloor + A|$ play an important role in governing the numerical behavior of $L$.  Due to the presence of the auxiliary divisor $A$, we call these ``twisted'' linear series.  In this section we recall the work of \cite{nakayama04} analyzing the asymptotic behavior of twisted linear series.

\begin{defn} \label{kappasigmadefinition}
Let $X$ be a normal variety, $L$ a pseudo-effective $\mathbb{R}$-divisor, and $A$ any divisor.
If $H^{0}(X,\mathcal{O}_{X}(\lfloor mL + A \rfloor))$ is non-zero for infinitely many values of $m$, we define
\begin{equation*}
\kappa_{\sigma}(L;A) := \max \left\{ k \in \mathbb{Z}_{\geq 0} \left| \limsup_{m \to \infty} \frac{h^{0}(X, \mathcal{O}_{X}(\lfloor mL + A \rfloor))}
{m^{k}} > 0 \right. \right\}.
\end{equation*}
Otherwise we define $\kappa_{\sigma}(L;A) = -\infty$.
The $\sigma$-dimension $\kappa_{\sigma}(X,L)$ is defined to be
\begin{equation*}
\kappa_{\sigma}(L) := \max_{A} \{ \kappa_{\sigma}(L;A) \}.
\end{equation*}
\end{defn}

Note that this maximum will be computed by some sufficiently ample divisor $A$.  Thus, we restrict our attention to the case when $A$ is an ample $\mathbb{Z}$-divisor from now on.

\begin{rmk} \label{ceilingremark} \label{floorremark}
As we increase $m$ the class of the divisor $\lceil mL \rceil - \lfloor mL \rfloor$ is bounded.
Thus if we replace $\lfloor - \rfloor$ by $\lceil - \rceil$ in the definition of $\kappa_{\sigma}(L)$, the
result is unchanged as the difference can be absorbed by the divisor $A$.
\end{rmk}

\begin{rmk}
\cite{nakayama04} asks whether $\kappa_{\sigma}(L)$ coincides with
\begin{itemize}
\item $\kappa_{\sigma}^{-}(L)$, where we replace the $\limsup$ by a $\liminf$.
\item $\kappa_{\sigma}^{+}(L)$, where we replace $>0$ by $< \infty$.
\end{itemize}
The equality of these three notions is a consequence of Theorem \ref{numdimproperties} (\ref{growthprop}).
\end{rmk}

\cite{nakayama04} shows that $\kappa_{\sigma}$ is a birational and numerical invariant.  In fact, since $\kappa_{\sigma}$ is one of the many equivalent definitions of the numerical dimension, it satisfies all of the properties of Theorem \ref{numdimproperties}.  The following key result shows that $\kappa_{\sigma}$ is non-negative for pseudo-effective divisors.

\begin{prop}[\cite{nakayama04}, Corollary V.1.4] \label{twistedsectionexistence}
Let $X$ be a smooth variety of dimension $n$.  Fix a big basepoint free divisor $B$ on $X$.  Then a divisor $L$ is pseudo-effective iff $h^{0}(X, \mathcal{O}_{X}(K_{X} + (n+2)B + \lceil mL \rceil)) > 0$ for every $m \geq 0$.
\end{prop}

\begin{proof}
\cite{nakayama04} Corollary V.1.4 shows a similar statement for $B$ very ample.  We explain how to extend the argument to the case when $B$ is big and basepoint free.  The main point is to show that there is an effective divisor $D \equiv (n+1)B + \lceil mL \rceil$ such that $\mathcal{J}(D)$ has an isolated point.  There is an effective divisor $E \equiv B + \lceil mL \rceil$.  Choose a general point $x$ that does not lie in $\Supp(E) \cup \mathbf{B}_{+}(B)$.  Let $B_{1},\ldots,B_{n^{2}} \in |B|$ be irreducible smooth divisors going through $x$.  Since $B$ is big, by choosing the $B_{i}$ sufficiently general we may ensure the intersections of any collection of at most $n$ of them has  the expected dimension.  Thus $D := \sum \frac{1}{n} B_{i} + E$ has multiplicity $n$ at $x$ and less than $1$ in a neighborhood of $x$.  By \cite{lazarsfeld04} Propositions 9.3.2 and 9.5.13, $\mathcal{J}(D)$ has an isolated point.  The proof then proceeds as in \cite{nakayama04} Corollary V.1.4.
\end{proof}

\section{Divisorial Zariski Decomposition} \label{divzardecomsection}

The divisorial Zariski decomposition is a higher-dimension analogue of the classical Zariski decomposition on surfaces.  It was introduced by \cite{nakayama04} and by \cite{boucksom04} in the analytic setting.

\begin{defn}
Let $X$ be a smooth variety and let $L$ be a pseudo-effective divisor.  Fix an ample divisor $A$ on $X$.  For any prime divisor $\Gamma$ on $X$ we define
\begin{equation*}
\sigma_{\Gamma}(L) = \lim_{\epsilon \to 0^{+}} \inf\{\mathrm{mult}_{\Gamma}(L') | L' \sim_{\mathbb{R}} L + \epsilon A \textrm{ and } L' \geq 0 \}.
\end{equation*}
By \cite{nakayama04} Lemma III.1.5 this definition is independent of the choice of $A$.
\end{defn}

\cite{nakayama04} Lemma III.1.7 shows that for any pseudo-effective divisor $L$ there are only finitely many prime divisors $\Gamma$ with $\sigma_{\Gamma}(L) > 0$.  Thus, we can define:

\begin{defn} \label{divzardecdef}
Let $X$ be a smooth variety, $L$ a pseudo-effective divisor.
We define:
\begin{equation*}
N_{\sigma}(L) = \sum \sigma_{\Gamma}(L) \Gamma \qquad \qquad P_{\sigma}(L) = L - N_{\sigma}(L)
\end{equation*}
The decomposition $L = N_{\sigma}(L) + P_{\sigma}(L)$ is called the divisorial Zariski
decomposition of $L$.
\end{defn}

The following proposition records the basic properties of the divisorial Zariski decomposition.  The key point is that $P_{\sigma}(L)$ captures all of the interesting geometric information about $L$.

\begin{prop}[\cite{nakayama04}, Lemma III.1.4, Corollary III.1.9, Theorem V.1.3]
Let $X$ be a smooth variety, $L$ a pseudo-effective divisor.
\begin{enumerate}
\item $N_{\sigma}(L)$ depends only on the numerical class of $L$.
\item $N_{\sigma}(L) \geq 0$ and $\kappa(N_{\sigma}(L)) = 0$.
\item $\Supp(N_{\sigma}(L))$ is precisely the divisorial part of $\mathbf{B}_{-}(L)$.
\item $H^{0}(X,\mathcal{O}_{X}(\lfloor mP_{\sigma}(L) \rfloor)) \to H^{0}(X,\mathcal{O}_{X}(\lfloor mL \rfloor))$ is an isomorphism for every $m \geq 0$.
\end{enumerate}
\end{prop}

Note that $N_{\sigma}(L) = 0$ iff $\mathbf{B}_{-}(L)$ has no divisorial components.  This simple observation leads to a different perspective on the divisorial Zariski decomposition.

\begin{defn}
Let $X$ be a smooth variety.  The movable cone $\overline{Mov}^{1}(X) \subset CD(X)$ is the cone consisting of the classes of all pseudo-effective divisors $L$ such that $\mathbf{B}_{-}(L)$ has no divisorial components.
\end{defn}

The positive part $P_{\sigma}(L)$ of the divisorial Zariski decomposition can be understood as a ``projection'' of $L$ onto the movable cone.  We will need a slightly modified version of \cite{nakayama04} Proposition III.1.14 that takes into account a subvariety $V$.

\begin{prop} \label{vbignessandpsigma}
Let $X$ be smooth, $V$ a subvariety, and $L$ a $V$-pseudo-effective divisor.  If $M$ is a movable divisor then $L \geq_{V} M$ iff $P_{\sigma}(L) \geq_{V} M$.  Thus $L - M$ is $V$-big (resp.~$V$-pseudo-effective) iff $P_{\sigma}(L)-M$ is $V$-big (resp.~$V$-pseudo-effective).
\end{prop}

\begin{proof}
First suppose that $P_{\sigma}(L) \geq_{V} M$.  Since $L$ is $V$-pseudo-effective, no component of $N_{\sigma}(L)$ contains $V$.  Thus $L \geq_{V} M$.  Conversely, suppose $L = M + E$ with $E \geq_{V} 0$.  Since $M$ is movable, $N_{\sigma}(L) \leq E$ by \cite{nakayama04} Proposition III.1.14.  Thus $E - N_{\sigma}(L)$ is still effective and does not contain $V$ in its support, showing that $P_{\sigma}(L) \geq_{V} M$.

Suppose now that $L-M$ is $V$-big.  Choose an ample divisor $A$ sufficiently small so that $L-M-A$ is $V$-big.  By Lemma \ref{vbignesscondition}, there is some $D \sim_{\mathbb{R}} L-M-A$ such that $D \geq_{V} 0$.  Applying the first step to $L-D$ shows that $P_{\sigma}(L) - L + D \equiv P_{\sigma} - M - A$ is $V$-pseudo-effective, so that $P_{\sigma}(L) - M$ is $V$-big.  The converse is straightforward.  The analogous statement for $V$-pseudo-effectiveness follows by taking limits.
\end{proof}

\subsection{Birational Properties}

Although the divisorial Zariski decomposition is not a birational invariant, its birational behavior is relatively nice.

\begin{prop}[\cite{nakayama04}, Theorem III.5.16]
Let $\phi: Y \to X$ be a birational map of smooth varieties and let $L$ be a pseudo-effective divisor on $X$.  Then $N_{\sigma}(\phi^{*}L) - \phi^{*}N_{\sigma}(L)$ is effective and $\phi$-exceptional.
\end{prop}

$L$ is said to admit a Zariski decomposition if there is a birational map $\phi: Y \to X$ from a smooth variety $Y$ such that $P_{\sigma}(\phi^{*}L)$ is nef.
An important example due to Nakayama (\cite{nakayama04}, Section IV.2) shows that Zariski decompositions do not always exist.  Nevertheless, there is a sense in which the positive part $P_{\sigma}(\phi^{*}L)$ becomes ``more nef'' as we pass to higher models $\phi: Y \to X$.  We will give two versions of this fact.  In the first, we consider a $V$-big divisor $L$. 

\begin{prop} \label{psigmaapproachesnef}
Let $X$ be smooth, $V$ a subvariety, and $L$ a $V$-big divisor with $L \geq_{V} 0$.  Then there is an effective divisor $G$ so that for any sufficiently large $m$ there is a model $\phi_{m}: \widetilde{X}_{m} \to X$ centered in $\mathbf{B}_{+}(L)$ and a big and nef divisor $N_{m}$ on $\widetilde{X}_{m}$ with
\begin{equation*}
N_{m} \leq_{\widetilde{V}_{m}} P_{\sigma}(\phi_{m}^{*}L) \leq_{\widetilde{V}_{m}} N_{m}+\frac{1}{m}\phi_{m}^{*}G
\end{equation*}
where $\widetilde{V}_{m}$ denotes the strict transform of $V$ on $\widetilde{X}_{m}$.
\end{prop}

The second version handles $V$-pseudo-effective divisors $L$.  Although the statement is slightly more technical, the additional flexibility will be useful later on.

\begin{prop} \label{psigmaapproachesnef2}
Let $X$ be smooth and let $L$ be a pseudo-effective divisor.  There are birational maps $\phi_{m}: \widetilde{X}_{m} \to X$ centered in $\mathbf{B}_{-}(L)$, an ample $\mathbb{Z}$-divisor $A$, and an effective divisor $G$ satisfying the following condition.  Suppose that $V$ is a subvariety of $X$ not contained in $\mathbf{B}_{-}(L)$.  Then there is some $G_{V} \sim_{\mathbb{Q}} G$ and for every $m$ there is an effective divisor $D_{m} \sim \lceil mL \rceil + A$ and a big and nef divisor $M_{m,D_{m}}$ such that
\begin{equation*}
M_{m,D_{m}} \leq_{\widetilde{V}_{m}} P_{\sigma}(\phi_{m}^{*}D_{m}) \leq_{\widetilde{V}_{m}} M_{m,D_{m}} + \phi_{m}^{*}G_{V}
\end{equation*}
where $\widetilde{V}_{m}$ denotes the strict transform of $V$ on $\widetilde{X}_{m}$.  We may furthermore assume that $A + D$ is ample for every $D$ supported on $\Supp(L)$ with coefficients in the set $[-3,3]$.
\end{prop}

Proposition \ref{psigmaapproachesnef} is equivalent to the following comparison between asymptotic multiplier ideals and base loci.  It is the analogue for $\mathbb{R}$-divisors of \cite{lazarsfeld04}, Theorem 11.2.21.  Note that the theory of asymptotic multiplier ideals for big $\mathbb{R}$-divisors works just as in the case of $\mathbb{Q}$-divisors.

\begin{lem} \label{multiplieridealbaselocuscomparison}
Let $X$ be smooth and let $L$ be a big divisor on $X$.  Fix a very ample $\mathbb{Z}$-divisor $H$ on $X$ such that $H + D$ is ample for every divisor $D$ supported on $\Supp(L)$ with coefficients in the set $[-3,3]$.  Suppose that $b$ is a sufficiently large positive integer so that $\lfloor bL \rfloor - (K_{X} + (n+1)H)$ is numerically equivalent to an effective $\mathbb{Z}$-divisor $G$.  Then for every $m \geq b$ we have
\begin{equation*}
\mathcal{J}(\Vert mL \Vert) \otimes \mathcal{O}_{X}(-G) \subseteq \mathfrak{b}(|\lfloor mL \rfloor|).
\end{equation*}
\end{lem}

\begin{proof}
The condition on $H$ guarantees that for $m \geq b$ we can write
\begin{align*}
\lfloor mL \rfloor - G & \equiv \lfloor mL \rfloor - \lfloor bL \rfloor + K_{X} + (n+1) H \\
& \equiv \left( (m-b)L + A \right) + K_{X} + nH
\end{align*}
for some ample $\mathbb{R}$-divisor $A$.  By applying Nadel vanishing and Castelnuovo-Mumford regularity, we find that
\begin{equation*}
\mathcal{O}_{X}(\lfloor mL \rfloor) \otimes \left( \mathcal{O}_{X}(-G) \otimes \mathcal{J}(\Vert (m-b)L \Vert) \right)
\end{equation*}
is globally generated for $m \geq b$.  Since $\mathcal{J}(\Vert mL \Vert) \subset \mathcal{J}(\Vert (m-b)L \Vert)$, this proves the theorem.
\end{proof}

\begin{proof}[Proof of Proposition \ref{psigmaapproachesnef}:]
Fix a very ample $\mathbb{Z}$-divisor $H$ and an integer $b$ as in Lemma \ref{multiplieridealbaselocuscomparison}.  Thus, for any $m \geq b$ we have
\begin{equation*}
\mathcal{J}(\Vert mL \Vert) \otimes \mathcal{O}_{X}(-G) \subseteq \mathfrak{b}(|\lfloor mL \rfloor|).
\end{equation*}
Recall that $G$ can be chosen to be any effective $\mathbb{Z}$-divisor numerically equivalent to $\lfloor bL \rfloor - (K_{X} + (n+1)H)$.  In particular, for $b$ large enough, the base locus of $|G|$ is contained in $\mathbf{B}_{+}(L)$.  Since this set does not contain $V$, we may ensure that $G \geq_{V} 0$.

Let $\phi_{m}: \widetilde{X}_{m} \to X$ be a resolution of the ideals $\mathfrak{b}(|\lfloor mL \rfloor|)$ and $\mathcal{J}(\Vert mL \Vert)$.  Note that each $\phi_{m}$ is centered in $\mathbf{B}_{+}(L)$.  We write $\phi_{m}^{-1}\mathfrak{b}(|\lfloor mL \rfloor|) \cdot \mathcal{O}_{Y_{m}} = \mathcal{O}_{Y_{m}}(-E_{m})$ and $\phi_{m}^{-1}\mathcal{J}(\Vert mL \Vert) \cdot \mathcal{O}_{Y_{m}} = \mathcal{O}_{Y_{m}}(-F_{m})$.  We also define the big and nef divisor $M_{m} := m\phi_{m}^{*}L - E_{m} - \phi_{m}^{*}\{ mL \}$.

We know that $F_{m}+\phi_{m}^{*}G \geq E_{m}$ for all sufficiently large $m$.  Let $M = \sum_{D \subset \Supp(L)} D$ be the sum of the components of $\Supp(L)$.  Replacing $G$ by $G + M$ allows us to take into account the fractional part of $mL$ so that
\begin{equation*}
F_{m}+\phi_{m}^{*}G \geq E_{m} + \phi_{m}^{*} \{ mL \}.
\end{equation*}
Note that still $G \geq_{V} 0$.  Since $L$ is $V$-big we know that $F_{m} \geq_{\widetilde{V}_{m}} 0$.   Thus, the inequality in the equation above is a $\widetilde{V}_{m}$-inequality.
Furthermore $N_{\sigma}(m\phi_{m}^{*}L) \geq_{\widetilde{V}_{m}} F_{m}$ by \cite{elmnp05} Proposition 2.5.  In all, we get $P_{\sigma}(m\phi_{m}^{*}L) \leq_{\widetilde{V}_{m}} M_{m} + \phi_{m}^{*}G$.  Dividing by $m$ and setting $N_{m} := M_{m}/m$ yields $P_{\sigma}(\phi_{m}^{*}L) \leq_{\widetilde{V}_{m}} N_{m} + \frac{1}{m}\phi_{m}^{*}G$.  The inequality $N_{m} \leq_{\widetilde{V}_{m}} P_{\sigma}(\phi_{m}^{*}L)$ follows from Proposition \ref{vbignessandpsigma} and the fact that $E_{m} + \phi_{m}^{*}\{ mL \} \geq_{\widetilde{V}_{m}} 0$.
\end{proof}

\begin{proof}[Proof of Proposition \ref{psigmaapproachesnef2}:]
Fix very ample divisors $H$ and $G$.
By Theorem \ref{restrictedbaselocuscountable}, there is an ample $\mathbb{Z}$-divisor $A$ such that $\textrm{Bs}(|\lceil mL \rceil + A|) \subset \mathbf{B}_{-}(L)$ for every positive integer $m$.  We may assume that $A$ is sufficiently ample so that:
\begin{itemize}
\item $\lceil mL \rceil + A - K_{X} - (n+1)H$ is numerically equivalent to an effective divisor $G_{m}$ for every $m>0$, and
\item $A+D$ is ample for every $D$ supported on $\Supp(L)$ with coefficients in the set $[-3,3]$.
\end{itemize}
Choose $D_{m} \sim \lceil mL \rceil + A$ so that $D_{m} \geq_{V} 0$.  Note that we can apply Proposition \ref{psigmaapproachesnef} to $D_{m}$ using $G_{m}$ as our choice of effective divisor (since $D_{m}$ is an integral divisor, there is no need to set conditions on the ampleness of $H$ along the components of $D_{m}$).  In particular, for every positive integer $m$ choose an $\epsilon_{m}>0$ such that $G-\epsilon_{m}G_{m}$ is ample.  Proposition \ref{psigmaapproachesnef} constructs a birational map $\phi_{m}: X_{m} \to X$ and big and nef divisors $M_{m,D_{m}}$ such that
\begin{equation*}
M_{m,D_{m}} \leq_{\widetilde{V}_{m}} P_{\sigma}(\phi_{m}^{*}D_{m}) \leq_{\widetilde{V}_{m}} M_{m,D_{m}} + \epsilon_{m}\phi_{m}^{*}G_{m}
\end{equation*}
Since $G - \epsilon_{m}G_{m}$ is $V$-big, we may replace $G$ by some $\mathbb{Q}$-linearly equivalent divisor $G_{V}$ so that
\begin{equation*}
M_{m,D_{m}} \leq_{\widetilde{V}_{m}} P_{\sigma}(\phi_{m}^{*}D_{m}) \leq_{\widetilde{V}_{m}} M_{m,D_{m}} + \phi_{m}^{*}G_{V}.
\end{equation*}
\end{proof}

\section{The Restricted Positive Product} \label{restrictedproductsection}

Fujita realized that one can study the asymptotic behavior of sections of a big divisor $L$ by analyzing the ample divisors sitting beneath $L$ on higher birational models.  The positive product (developed in \cite{boucksom04} and \cite{bdpp04}) is a construction that encapsulates this approach to asymptotic behavior.

In this section we will discuss the restricted positive product $\langle L_{1} \cdot L_{2} \cdot \ldots \cdot L_{k} \rangle_{X|V}$ of \cite{bfj09}.  In contrast to the usual intersection product $L_{1} \cdot L_{2} \cdot \ldots \cdot L_{k} \cdot V$, the restricted positive product throws away the contributions of the base loci of the $L_{i}$.  The result is a numerical equivalence class of cycles on $V$ that gives a more precise measure of the positivity of the $L_{i}$ along $V$.

\subsection{Definition and Basic Properties}

In this section we review the construction of the restricted positive product in \cite{bfj09}.  Throughout we will use the intersection product of \cite{fulton84}.  We will use the following notation:

\begin{defn}
Let $X$ be a normal variety.  Suppose that $V$ is a subvariety of $X$ and that $[L] \in CD(X)$.  We will let $[L]|_{V}$ denote the image under the restriction map $CD^{1}(X) \to CD^{1}(V)$.
\end{defn}

Note that if $L$ is a divisor such that $\Supp(L) \not \supset V$ then $[L|_{V}] = [L]|_{V}$.

\begin{defn}
Let $X$ be a normal variety variety of dimension $n$.  Suppose that $K$, $K'$ are two classes in $N^{k}(X)$.  We write $K \succeq K'$ if $K-K'$ is contained in the closure of the cone generated by effective cycles of dimension $n-k$.
\end{defn}

We will often use the following basic lemma.

\begin{lem}[\cite{bfj09}, Proposition 2.3 and Definition 4.4] \label{nefcomparison}
Let $X$ be a smooth variety and let $V$ be a subvariety of $X$.  Suppose that $N_{1},\ldots,N_{k}$ and $N_{1}',\ldots,N_{k}'$ are nef divisors on $X$ satisfying $N_{i} \geq_{V} N_{i}'$.  Then
\begin{equation*}
N_{1} \cdot \ldots \cdot N_{k} \cdot V \succeq N_{1}' \cdot \ldots \cdot N_{k}' \cdot V.
\end{equation*}
\end{lem}

\begin{thrm}[\cite{bfj09},Lemma 2.6 and Lemma 2.7]
Let $X$ be a normal variety, $V$ a subvariety not contained in $\Sing(X)$, and $L_{1},\ldots,L_{k}$ $V$-big divisors.  Consider the classes
\begin{equation*}
\phi_{*}(N_{1} \cdot N_{2} \cdot \ldots \cdot N_{k} \cdot \widetilde{V}) \in N^{k}(V)
\end{equation*}
where $\phi: (\widetilde{X},\widetilde{V}) \to (X,V)$ varies over all smooth $V$-birational models, the $N_{i}$ are nef, and $E_{i} := \phi^{*}L_{i} - N_{i}$ is a $\mathbb{Q}$-divisor satisfying $E_{i} \geq_{\widetilde{V}} 0$.  These classes form a directed set under the relation $\preceq$ and admit a unique maximum under this relation.
\end{thrm}

\begin{rmk}
Although \cite{bfj09} only proves this when $V$ is a prime divisor in $X$, the proof works without change in this more general situation.
\end{rmk}

The restricted positive product is defined as the maximum class occurring in the previous theorem.

\begin{defn} \label{restrictedpositiveproductdefn}
Let $X$ be a normal variety and let $V$ be a subvariety not contained in $\Sing(X)$.  Let $L_{1},L_{2},\ldots,L_{k}$ be $V$-big divisors.  We define the cycle
\begin{equation*}
\langle L_{1} \cdot L_{2} \cdot \ldots \cdot L_{k} \rangle_{X|V} \in N^{k}(V)
\end{equation*}
to be the maximum under $\preceq$ over all smooth $V$-birational models $\phi: (\widetilde{X},\widetilde{V}) \to (X,V)$ of
\begin{equation*}
\phi_{*}(N_{1} \cdot N_{2} \cdot \ldots \cdot N_{k} \cdot \widetilde{V})
\end{equation*}
where the $N_{i}$ are nef and $E_{i} := \phi^{*}L_{i} - N_{i}$ is a $\mathbb{Q}$-divisor satisfying $E_{i} \geq_{\widetilde{V}} 0$.  In the special case $X = V$, we write $\langle L_{1} \cdot L_{2} \cdot \ldots \cdot L_{k} \rangle_{X}$.
\end{defn}

In fact, \cite{bfj09} Proposition 2.13 shows that the definition is unchanged if we allow $E_{i}$ to be a $V$-pseudo-effective $\mathbb{R}$-divisor.  The restricted positive product satisfies a number of important properties.

\begin{prop}[\cite{bfj09},Proposition 4.6]
As a function on the $k$-fold product of the $V$-big cone, the restricted positive product is continuous, symmetric, homogeneous of degree $1$, and super-additive in each variable in the sense that
\begin{equation*}
\langle (L + L') \cdot L_{2} \cdot \ldots \cdot L_{k} \rangle_{X|V} \succeq \langle L \cdot L_{2} \cdot \ldots \cdot L_{k} \rangle_{X|V} + \langle L' \cdot L_{2} \cdot \ldots \cdot L_{k} \rangle_{X|V}.
\end{equation*}
\end{prop}

Since the product is continuous, this allows us to define a limit as we approach the pseudo-effective cone.

\begin{defn}
Let $X$ be a normal variety, $V$ a subvariety not contained in $\Sing(X)$, and $L_{1},L_{2},\ldots,L_{k}$ $V$-pseudo-effective divisors.  For each $i$ fix a sequence of $V$-big divisors $B_{i,j}$ converging to $0$ as $j$ increases.  We define the class
\begin{equation*}
\langle L_{1} \cdot L_{2} \cdot \ldots \cdot L_{k} \rangle_{X|V} = \lim_{j \to \infty} \langle (L_{1} + B_{1,j}) \cdot (L_{2} + B_{2,j}) \cdot \ldots \cdot (L_{k} + B_{k,j}) \rangle_{X|V}.
\end{equation*}
Note that this limit is independent of the choice of the $B_{i,j}$ since by super-additivity any two choices are comparable under $\succeq$.
\end{defn}

We will sometimes abuse notation by allowing the restricted positive product to take numerical classes as arguments rather than actual divisors.  Since the restricted positive product is compatible under pushforward, we can extend the definition to arbitrarily singular varieties in the following way.

\begin{defn}
Let $X$ be an integral variety and let $\phi: Y \to X$ be a smooth model.  For $[L_{1}],\ldots,[L_{k}] \in CD(X)$ we define
\begin{equation*}
\langle [L_{1}] \cdot \ldots \cdot [L_{k}] \rangle_{X} := \phi_{*} \langle \phi^{*}[L_{1}] \cdot \ldots \cdot \phi^{*}[L_{k}] \rangle_{Y}.
\end{equation*}
\end{defn}

Even though the restricted positive product is continuous along the $V$-big cone, it is only semi-continuous along the $V$-pseudo-effective boundary in the sense that if $L_{i,j}$ is a sequence of $V$-pseudo-effective divisors whose limit is $L_{i}$ then
\begin{equation*}
\langle L_{1} \cdot \ldots \cdot L_{k} \rangle_{X|V} \succeq \lim_{j \to \infty} \langle L_{1,j} \cdot \ldots \cdot L_{k,j} \rangle_{X|V}
\end{equation*}

As noted in \cite{bfj09}, it is most natural to consider the restricted positive product as the set of classes $\{ \langle \phi^{*}L_{1} \cdot \ldots \cdot \phi^{*}L_{k} \rangle_{\widetilde{X}|\widetilde{V}} \}$ on all smooth $V$-birational models $\phi: \widetilde{X} \to X$, or in other words, as a class on the Riemann-Zariski space of $V$.  Although we will not develop this principle systematically, this idea appears implicitly as some theorems will only hold upon taking a limit over all sufficiently high birational models.

Since the restricted positive product should be considered as a birational object, the class in $N^{k}(V)$ may not be closely related to the geometry of $L$ and $V$.  The class $\langle L_{1} \cdot \ldots \cdot L_{k} \rangle_{X|V}$ seems to be most interesting in the following two situations.

\begin{exmple}
When $X$ is smooth $\langle L \rangle_{X}$ is the numerical class of $P_{\sigma}(L)$.  It suffices to check this when $L$ is big.  Recall that for any birational map $\phi: Y \to X$ from a smooth variety $Y$ we have $\phi_{*}P_{\sigma}(\phi^{*}L) = P_{\sigma}(L)$.  Thus, choosing an effective divisor $G$ as in Proposition \ref{psigmaapproachesnef}, the result of the proposition implies that for any $\epsilon > 0$ we have $\langle L \rangle_{X} \preceq [P_{\sigma}(L)] \preceq \langle L + \epsilon G \rangle_{X}$.  Letting $\epsilon \to 0$ demonstrates the equality.
\end{exmple}

\begin{exmple}
Consider $\langle L_{1} \cdot \ldots \cdot L_{d} \rangle_{X|V}$ where $d = \dim V$.  Since the restricted positive product is compatible under pushforward, $\deg \langle \phi^{*}L_{1} \cdot \ldots \cdot \phi^{*}L_{d} \rangle_{\widetilde{X}|\widetilde{V}}$ is independent of the choice of $V$-birational model $(\widetilde{X},\widetilde{V})$ by the projection formula.  In fact, we have

\begin{prop}[\cite{elmnp06}, Proposition 2.11 and Theorem 2.13] \label{restrictedvolumeandproduct}
Let $X$ be a smooth variety, $V$ a $d$-dimensional subvariety, and $L$ a $V$-big divisor.  Then $\deg \langle L^{d} \rangle_{X|V} = \vol_{X|V}(L)$.
\end{prop}

\end{exmple}

\subsection{Properties of the Restricted Positive Product}

In this section we study the properties of the restricted positive product.  The main goal of the section is to show that the restricted positive product can be interpreted as the usual intersection product of $P_{\sigma}(\phi^{*}L_{i})$ if we take a limit over all birational models $\phi$.  The advantage of this viewpoint is that it gives us a natural interpretation of the restricted positive product along the boundary of the pseudo-effective cone.

We first show that the restricted positive product has a natural compatibility with the divisorial Zariski decomposition.

\begin{prop} \label{positiveproductpositivepart}
Let $X$ be a smooth variety, $V$ a subvariety, and  $L_{1},\ldots,L_{k}$ $V$-pseudo-effective divisors.  Then
\begin{equation*}
\langle L_{1} \cdot \ldots \cdot L_{k} \rangle_{X|V} = \langle P_{\sigma}(L_{1}) \cdot \ldots \cdot P_{\sigma}(L_{k}) \rangle_{X|V}.
\end{equation*}
\end{prop}

\begin{proof}
First suppose that the $L_{i}$ are $V$-big.  Since any nef divisor is movable, Proposition \ref{vbignessandpsigma} shows that for any of the $N_{i}$ as in Definition \ref{restrictedpositiveproductdefn} we have
$P_{\sigma}(\phi^{*}L_{i}) \geq_{\widetilde{V}} N_{i}$.
We also know that $N_{\sigma}(\phi^{*}L_{i}) \geq_{\widetilde{V}} \phi^{*}N_{\sigma}(L_{i})$ since $V$ is not contained in $\mathbf{B}_{-}(L_{i})$.  Combining the two inequalities yields
\begin{equation*}
\phi^{*}P_{\sigma}(L_{i}) \geq_{\widetilde{V}} N_{i}.
\end{equation*}
Thus the classes $\langle L_{1} \cdot \ldots \cdot L_{k} \rangle_{X|V}$ and $\langle P_{\sigma}(L_{1}) \cdot \ldots \cdot P_{\sigma}(L_{k}) \rangle_{X|V}$ are computed by taking a maximum over the same sets, showing that they are equal.

Now suppose that the $L_{i}$ are only $V$-pseudo-effective.  Fix an ample divisor $A$ on $X$.  Note that
\begin{equation*}
P_{\sigma}(L + \epsilon A) - P_{\sigma}(L) = \epsilon A + (N_{\sigma}(L) - N_{\sigma}(L + \epsilon A))
\end{equation*}
is $V$-big.  As $\epsilon$ goes to $0$ these $V$-big classes also converge to $0$.  Thus
\begin{equation*}
\langle P_{\sigma}(L_{1}) \cdot \ldots \cdot P_{\sigma}(L_{k}) \rangle_{X|V} = \lim_{\epsilon \to 0} \langle P_{\sigma}(L_{1} + \epsilon A) \cdot \ldots \cdot P_{\sigma}(L_{k} + \epsilon A) \rangle_{X|V}.
\end{equation*}
Applying the $V$-big case to the right-hand side finishes the proof.
\end{proof}

The following proposition of \cite{bfj09} compares the restricted positive product of the $L_{i}$ along $V$ with the positive product of the restrictions $L_{i}|_{V}$.  \cite{bfj09} only proves the statement when the $L_{i}$ are $V$-big, but the proposition extends to the $V$-pseudo-effective case by taking limits.

\begin{prop}[\cite{bfj09}, Remark 4.5] \label{comparingrestpospodrestrictions}
Let $X$ be a smooth variety, $V$ a subvariety, and $L_{1},\ldots,L_{k}$ $V$-pseudo-effective divisors.  Then
\begin{equation*}
\langle L_{1} \cdot \ldots \cdot L_{k} \rangle_{X|V} \preceq \langle [L_{1}]|_{V} \cdot \ldots \cdot [L_{k}]|_{V} \rangle_{V}.
\end{equation*}
\end{prop}

By combining Propositions \ref{positiveproductpositivepart} and \ref{comparingrestpospodrestrictions} we obtain
\begin{equation*}
\langle L_{1} \cdot \ldots \cdot L_{k} \rangle_{X|V} \preceq \phi_{*} \langle [P_{\sigma}(\phi^{*}L_{1})]|_{\widetilde{V}} \cdot \ldots \cdot [P_{\sigma}(\phi^{*}L_{k})]|_{\widetilde{V}} \rangle_{\widetilde{V}}
\end{equation*}
where $\phi: (\widetilde{X},\widetilde{V}) \to (X,V)$ is any $V$-birational model.  The main theorem of this section states that by taking a limit over all birational models the right-hand side approaches the left.

\begin{thrm} \label{positiveproductpsigmaalternative}
Let $X$ be a smooth variety, $V$ a subvariety, and $L_{1},\ldots,L_{k}$ $V$-pseudo-effective divisors.  Fix an ample divisor $A$.  Then for any $\epsilon$ there is some $V$-birational map $\phi: (\widetilde{X},\widetilde{V}) \to (X,V)$ such that
\begin{equation*}
\phi_{*} \langle [P_{\sigma}(\phi^{*}L_{1})]|_{\widetilde{V}} \cdot \ldots \cdot [P_{\sigma}(\phi^{*}L_{k})]|_{\widetilde{V}} \rangle_{\widetilde{V}} \preceq \langle L_{1} \cdot \ldots \cdot L_{k} \rangle_{X|V} + \epsilon A^{k} \cdot V.
\end{equation*}
\end{thrm}

\begin{proof}
First suppose the $L_{i}$ are $V$-big.  By Lemma \ref{vbignesscondition} we may replace the $L_{i}$ by some $\mathbb{R}$-linearly equivalent divisors to ensure that $L_{i} \geq_{V} 0$.  Proposition \ref{psigmaapproachesnef} then yields an effective divisor $G_{i}$ such that for any $m$ there is a $V$-birational model $\phi: \widetilde{X}_{m} \to X$ with
\begin{equation*}
N_{m,i} \leq_{\widetilde{V}} P_{\sigma}(\phi_{m}^{*}L_{i}) \leq_{\widetilde{V}} N_{m,i}+\frac{1}{m}\phi_{m}^{*}G_{i}
\end{equation*}
for some nef divisors $N_{m,i}$.  Fix some ample divisor $A$ on $X$ such that $A - L_{i}$ and $A - G_{i}$ are ample for every $i$.  By Lemma \ref{nefcomparison} there is some constant $C$ such that
\begin{equation*}
\phi_{m*}\langle [P_{\sigma}(\phi_{m}^{*}L_{1})]|_{\widetilde{V_{m}}} \cdot \ldots \cdot [P_{\sigma}(\phi_{m}^{*}L_{k})]|_{\widetilde{V_{m}}} \rangle_{\widetilde{V_{m}}} \preceq \phi_{*}(N_{m,1} \cdot \ldots \cdot N_{m,k} \cdot \widetilde{V}) + \frac{C}{m} A^{k} \cdot V.
\end{equation*}

Now suppose that the $L_{i}$ are only $V$-pseudo-effective.  We first choose an ample divisor $H$ so that
\begin{equation*}
\langle (L_{1} + H) \cdot \ldots \cdot (L_{k} + H) \rangle_{X|V} \preceq \langle L_{1} \cdot \ldots \cdot L_{k} \rangle_{X|V} + \frac{\epsilon}{2} A^{k} \cdot V.
\end{equation*}
Construct a model $\phi$ by applying the $V$-big case to the $L_{i} + H$ and $\epsilon/2$.  Since $P_{\sigma}(\phi^{*}(L_{i}+H)) - P_{\sigma}(\phi^{*}L)$ is $\widetilde{V}$-pseudo-effective, the conclusion follows.
\end{proof}

\begin{cor} \label{positiveproductpsigmapsef}
Let $X$ be a smooth variety and let $L_{1},\ldots,L_{k}$ be pseudo-effective divisors.  There is a sequence of birational maps $\phi_{m}: X_{m} \to X$ centered in $\cup_{i} \mathbf{B}_{-}(L_{i})$ such that for any subvariety $V$ not contained in $\cup_{i} \mathbf{B}_{-}(L_{i})$ we have
\begin{equation*}
\langle L_{1} \cdot \ldots \cdot L_{k} \rangle_{X|V} = \lim_{m \to \infty} \phi_{m*} \langle [P_{\sigma}(\phi_{m}^{*}L_{1})]|_{\widetilde{V}_{m}} \cdot \ldots \cdot [P_{\sigma}(\phi_{m}^{*}L_{k})]|_{\widetilde{V}_{m}} \rangle_{\widetilde{V}_{m}}
\end{equation*}
\end{cor}

\begin{proof}
Fix a sequence of birational maps $\phi_{m}$, an ample divisor $A$, and an effective divisor $G$ as in Proposition \ref{psigmaapproachesnef2} for each of the $L_{i}$ simultaneously.  The proposition constructs divisors $D_{m,i} \equiv \lceil mL_{i} \rceil + A$ and big and nef divisors $M_{m,i,D_{m,i}}$ such that
\begin{equation*}
M_{m,i,D_{m,i}} \leq_{\widetilde{V}_{m}} P_{\sigma}(\phi_{m}^{*}D_{m,i}) \leq_{\widetilde{V}_{m}} M_{m,i,D_{m,i}} + \phi_{m}^{*}G_{V}.
\end{equation*}
Just as in the previous proposition we have
\begin{align*}
\lim_{m \to \infty} & \frac{1}{m^{k}} \phi_{m*} (M_{m,1,D_{m,1}} \cdot \ldots \cdot M_{m,k,D_{m,k}} \cdot \widetilde{V}_{m}) \\
&\preceq \lim_{m \to \infty} \frac{1}{m^{k}} \langle D_{m,1} \cdot \ldots \cdot D_{m,k} \rangle_{X|V} \\
& \preceq \lim_{m \to \infty} \frac{1}{m^{k}} \phi_{m*} \langle [P_{\sigma}(\phi_{m}^{*}D_{m,1})]|_{\widetilde{V}_{m}} \cdot \ldots \cdot  [P_{\sigma}(\phi_{m}^{*}D_{m,k})]|_{\widetilde{V}_{m}} \rangle_{\widetilde{V}_{m}} \\
& \preceq \lim_{m \to \infty} \frac{1}{m^{k}} \phi_{m*} ( (M_{m,1,D_{m,1}} + \phi_{m}^{*}G_{V}) \cdot \ldots \cdot (M_{m,k,D_{m,k}} + \phi_{m}^{*}G_{V}) \cdot \widetilde{V}_{m})
\end{align*}
Arguing as in the previous proof, we see that the leftmost and rightmost expressions converge as $m$ increases.  Recall that by our choice of $A$ we have $\lceil mL_{i} \rceil + A - mL_{i}$ is $V$-big for every $m$.  Thus
\begin{equation*}
\langle L_{1} \cdot \ldots \cdot L_{k} \rangle_{X|V} = \lim_{m \to \infty} \left\langle \frac{1}{m} D_{m,1} \cdot \ldots \cdot \frac{1}{m} D_{m,k} \right \rangle_{X|V}
\end{equation*}
so that the sequence converges to the restricted positive product as desired.
\end{proof}

We extract a useful feature of the previous arguments as a definition.

\begin{defn} \label{computeposproddef}
Let $X$ be a smooth variety, $V$ a subvariety, and $L_{1},\ldots,L_{k}$ $V$-big divisors.  Choose $L_{i}' \sim_{\mathbb{Q}} L_{i}$ satisfying $L_{i}' \geq_{V} 0$.  Suppose that $\phi_{m}$ is a countable sequence of maps that satisfy the conclusion of Proposition \ref{psigmaapproachesnef} for every $L_{i}'$ simultaneously.  We say that the $\phi_{m}$ compute the restricted positive product of the $L_{i}$.
\end{defn}

Note that for any finite set of subvarieties $V_{1},\ldots,V_{r}$ we can choose $\phi_{m}$ and $N_{m}$ to simultaneously compute the restricted positive product for each $V_{j}$.  The key property of Definition \ref{computeposproddef} is that only countably many maps are needed to compute the restricted positive product.  

The restricted positive product reduces to the usual product for nef divisors.

\begin{lem} \label{restrictedproductandnefdivisors}
Let $X$ be a smooth variety, $V$ a subvariety, and $L_{1},\ldots,L_{k}$ $V$-pseudo-effective divisors.
\begin{enumerate}
\item Suppose $N$ is a nef divisor.  Then
\begin{equation*}
\langle L_{1} \cdot L_{2} \cdot \ldots \cdot L_{k} \cdot N \rangle_{X|V} = \langle L_{1} \cdot L_{2} \cdot \ldots \cdot L_{k} \rangle_{X|V} \cdot N|_{V}.
\end{equation*}
\item If $H$ is a very general element of a basepoint free linear system, then
\begin{equation*}
\langle L_{1} \cdot L_{2} \cdot \ldots \cdot L_{k} \rangle_{X|V} \cdot H = \langle L_{1} \cdot L_{2} \cdot \ldots \cdot L_{k} \rangle_{X|V \cap H}.
\end{equation*}
\item If $f: X \to Z$ is a morphism and $F$ is a very general fiber then
\begin{equation*}
\langle L_{1} \cdot L_{2} \cdot \ldots \cdot L_{k} \rangle_{X|V} \cdot F = \langle L_{1} \cdot L_{2} \cdot \ldots \cdot L_{k} \rangle_{X|V \cap F}.
\end{equation*}
\end{enumerate}
\end{lem}

\begin{proof}
For each of these properties, it is enough to check the case when the $L_{i}$ are $V$-big.

The first property is shown in \cite{bfj09}, Proposition 4.7; one simply notes that for an ample divisor $A$ the pull-back $\phi^{*}A$ is already nef so that one may take $\phi^{*}A$ to be the nef divisor in Definition \ref{restrictedpositiveproductdefn}.  By taking limits as $A$ approaches $N$ we obtain the statement.

To show the second property, consider a countable set of smooth $V$-birational models $\phi_{m}: \widetilde{X}_{m} \to X$ that compute the restricted positive product.  Choose $H$ sufficiently general so that it does not contain any $\phi_{m}$-exceptional center.  Then the strict transform of $V \cap H$ is a cycle representing the class $\phi_{m}^{*}H \cdot \widetilde{V}$.  Thus we can identify the classes
\begin{align*}
\phi_{m*}(N_{1} \cdot N_{2} \cdot \ldots \cdot N_{k} \cdot \widetilde{V}) \cdot H & =
\phi_{m*}(N_{1} \cdot N_{2} \cdot \ldots \cdot N_{k} \cdot \phi_{m}^{*}H \cdot \widetilde{V}) \\
& = \phi_{m*}(N_{1} \cdot N_{2} \cdot \ldots \cdot N_{k} \cdot \widetilde{V \cap H})
\end{align*}

The third property can be proved by a similar argument.  One uses the second property inductively by pulling-back very ample divisors from $Z$.
\end{proof}

\begin{cor} \label{nonvanishingposprod}
Let $X$ be a normal variety, let $V$ be a subvariety not contained in $\Sing(X)$, and let $L_{1},\ldots,L_{k}$ be $V$-pseudo-effective divisors.  Suppose that $\phi: (\widetilde{X},\widetilde{V}) \to (X,V)$ is a  smooth $V$-birational model.  If $\langle \phi^{*}L_{1} \cdot \ldots \cdot \phi^{*}L_{k} \rangle_{\widetilde{X}|\widetilde{V}} \neq 0$, then $\langle L_{1} \cdot \ldots \cdot L_{k} \rangle_{X|V} \neq 0$.
\end{cor}

\begin{proof}
Let $A$ be an ample divisor on $\widetilde{X}$ and let $H$ be an ample divisor on $X$ such that $\phi^{*}H \geq A$.  Since $\phi$ is $V$-birational, we may ensure that $\Supp(\phi^{*}H - A)$ does not contain $\widetilde{V}$.  Setting $d = \dim V$, we have
\begin{align*}
\langle L_{1} \cdot \ldots \cdot L_{k} \rangle_{X|V} \cdot H^{d-k} & = \langle \phi^{*}L_{1} \cdot \ldots \cdot \phi^{*}L_{k} \rangle_{\widetilde{X}|\widetilde{V}} \cdot \phi^{*}H^{d-k} \\
& = \langle \phi^{*}L_{1} \cdot \ldots \cdot \phi^{*}L_{k} \cdot \phi^{*}H^{d-k} \rangle_{\widetilde{X}|\widetilde{V}} \\
& \geq \langle \phi^{*}L_{1} \cdot \ldots \cdot \phi^{*}L_{k} \cdot A^{d-k} \rangle_{\widetilde{X}|\widetilde{V}} \\
& = \langle \phi^{*}L_{1} \cdot \ldots \cdot \phi^{*}L_{k} \rangle_{\widetilde{X}|\widetilde{V}} \cdot A^{d-k} > 0.
\end{align*}
\end{proof}

We next consider how the restricted positive product behaves when passing to an admissible model.

\begin{prop} \label{restrictedpositiveproductadmissiblemodels}
Let $X$ be a smooth variety, $V$ a subvariety, and $L_{1},\ldots,L_{k}$ $V$-pseudo-effective divisors.  Suppose that $f: (Y,W) \to (X,V)$ is an admissible model.  Then
\begin{equation*}
f_{*} \langle f^{*}L_{1} \cdot \ldots \cdot f^{*}L_{k} \rangle_{Y|W} = \deg(f|_{W}) \langle L_{1} \cdot \ldots \cdot L_{k} \rangle_{X|V}.
\end{equation*}
\end{prop}

Note that $f^{*}L_{i}$ is $W$-pseudo-effective by Proposition \ref{restrictedbaselocusinvariance}.

\begin{proof}
It suffices to consider the case when the $L_{i}$ are $V$-big.  By Lemma \ref{vbignesscondition} we may suppose that $L_{i} \geq_{V} 0$.  Let $\phi_{m}: X_{m} \to X$ be a sequence of $V$-birational models that computes $\langle L_{1} \cdot \ldots \cdot L_{k} \rangle_{X|V}$ and let $\psi_{m}: Y_{m} \to Y$ be a sequence of $W$-birational models that computes $\langle f^{*}L_{1} \cdot \ldots \cdot f^{*}L_{k} \rangle_{Y|W}$.  Since the natural map $\phi_{m}^{-1} \circ f \circ \psi_{m}$ is a morphism on the generic point of $W$, by passing to higher $W$-birational models we may assume that $Y_{m}$ admits a morphism $f_{m}: Y_{m} \to X_{m}$.  Note that
\begin{equation*}
f_{m}^{*}N_{i,m} \leq_{\widetilde{W}_{m}} P_{\sigma}(\psi_{m}^{*}f_{m}^{*}L_{i}) \leq_{\widetilde{W}_{m}} f_{m}^{*}P_{\sigma}(\phi_{m}^{*}L_{i}) \leq_{\widetilde{W}_{m}} f_{m}^{*}N_{i,m} + \frac{1}{m}f_{m}^{*}\phi_{m}^{*}G_{i}.
\end{equation*}
By construction the pushforwards
\begin{equation*}
\phi_{m*}f_{m*}(f_{m}^{*}N_{1,m} \cdot \ldots \cdot f_{m}^{*}N_{k,m} \cdot \widetilde{W}_{m})
\end{equation*}
converge to $\deg(f|_{W})\langle L_{1} \cdot \ldots \cdot L_{k} \rangle_{X|V}$.  The same is true for the terms on the right hand side.  Thus $f_{*}\psi_{m*}\langle P_{\sigma}(\psi_{m}^{*}f^{*}L_{1}) \cdot \ldots \cdot  P_{\sigma}(\psi_{m}^{*}f^{*}L_{k}) \rangle_{Y|\widetilde{W}_{m}}$ converges to the same thing, and Proposition \ref{positiveproductpsigmaalternative} finishes the proof.  
\end{proof}

It is worth pointing out that Proposition \ref{restrictedpositiveproductadmissiblemodels} does not contradict the invariance of $\vol_{X|V}(L)$ under passing to admissible models.  Even if $L$ is $V$-big, $\phi^{*}L$ will not be $W$-big when $\deg(f|_{W}) > 1$, so Proposition \ref{restrictedvolumeandproduct} does not apply to $W$.

\begin{prop} \label{ampleintersectiondominates}
Let $X$ be a smooth variety, $V$ a subvariety of dimension $d$, and $L$ a $V$-pseudo-effective divisor.  Suppose that $\deg(\langle L^{d} \rangle_{X|V}) > 0$.  Then for a very general intersection of very ample divisors $W$ of dimension $d$ we also have $\deg(\langle L^{d} \rangle_{X|W}) > 0$.
\end{prop}

\begin{proof}
Fix a sequence of maps $\phi_{m}: \widetilde{X}_{m} \to X$ for $L$ as in Corollary \ref{positiveproductpsigmapsef}.  By choosing very ample divisors $H_{1},\ldots,H_{n-d}$ very general in their linear systems, we may ensure that no $H_{i}$ contains any $\phi_{m}$-exceptional center and the intersection $W = H_{1} \cap \ldots \cap H_{n-d}$ is smooth of the expected dimension.

For each $i = 1,2,\ldots,n-d$, choose a positive integer $c_{i}$ so that $\mathcal{I}_{V}(c_{i}H_{i})$ is generated by global sections and set $C = \prod_{i} c_{i}^{-1}$.  Note that for any $V$-birational model $\phi: (Y,\widetilde{V}) \to (X,V)$, there are $D_{i} \in |c_{i}\phi^{*}H_{i}|$ such that each $D_{i}$ has multiplicity at least $1$ along $\widetilde{V}$ and $D_{1} \cap \ldots \cap D_{n-k}$ has dimension $k$.  In particular for $\phi_{m}$ we have
\begin{align*}
[\widetilde{W}] & = C[\phi_{m}^{*}c_{1}H_{1}] \cap [\phi_{m}^{*}c_{2}H_{2}] \cap \ldots \cap [\phi_{m}^{*}c_{n-d}H_{n-d}] \\
& \succeq C[\widetilde{V}]
\end{align*}
where $\widetilde{W}$ and $\widetilde{V}$ denote the strict transforms of $W$ and $V$ on $\widetilde{X}_{m}$.  In particular, for any nef divisor $N$ on $\widetilde{X}_{m}$ we have $N^{d} \cdot \widetilde{W} \geq N^{d} \cdot \widetilde{V}$, and the conclusion follows.
\end{proof}

\section{Nakayama Constants} \label{nakayamaconstantsection}

Suppose that $L$ is an ample divisor and $V$ is a subvariety in $X$.   Let $\phi: Y \to X$ be a smooth resolution of the ideal $\mathcal{I}_{V}$ and define the divisor $E$ by the equation $\mathcal{O}_{Y}(-E) = \phi^{-1}\mathcal{I}_{V} \cdot \mathcal{O}_{Y}$.  The Seshadri constant
\begin{equation*}
\varepsilon(L,V) := \max \{ \, \tau \, | \, \phi^{*}L - \tau E \textrm{ is nef } \}
\end{equation*}
measures ``how ample'' $L$ is along the subvariety $V$.  Seshadri constants play an important role in understanding the positivity properties of ample divisors.  We will be interested in a related notion that can be defined for an arbitrary pseudo-effective divisor $L$.  It first appears in connection with the numerical dimension in \cite{nakayama04}.

\begin{defn}
Let $X$ be a normal variety, $\mathcal{I}$ be an ideal sheaf on $X$, and $L$ be a pseudo-effective divisor.   Choose a smooth resolution $\phi: Y \to X$ of $\mathcal{I}$ and define $E$ by setting $\mathcal{O}_{Y}(-E) = \phi^{-1}\mathcal{I} \cdot \mathcal{O}_{Y}$.  We define the Nakayama constant
\begin{equation*}
\varsigma(L,\mathcal{I}) := \max \{ \, \tau \, | \, \phi^{*}L - \tau E \textrm{ is pseudo-effective } \}.
\end{equation*}
Of course, $\varsigma$ is independent of the choice of resolution.  When $\mathcal{I}$ is the ideal sheaf of a subvariety $V$, we will also denote the Nakyama constant by $\varsigma(L,V)$.
\end{defn}

One advantage of $\varsigma(L,V)$ is that it can be positive even when $L$ is pseudo-effective but not big.  Thus the Nakayama constant is a more sensitive measure of positivity than the moving Seshadri constant of \cite{nakamaye03} which always vanishes as we approach the pseudo-effective boundary.  It turns out that the Nakayama constant is closely related to the other notions of positivity we have considered.

\begin{rmk} \label{equivalenttonakayama}
\cite{nakayama04} works with a slightly different formulation of this concept.  Nakayama's definition is equivalent to ours; the equivalence is demonstrated in the first paragraph of the proof of Proposition \ref{nakayamasformulation}.
\end{rmk}

There is a useful criterion for non-vanishing of $\varsigma$ which is closer in spirit to Nakayama's original formulation.

\begin{prop} \label{nakayamasformulation}
Let $X$ be a normal variety, $\mathcal{I}$ be an ideal sheaf, and $L$ be a pseudo-effective divisor.  Then $\varsigma(L,\mathcal{I}) > 0$ iff there is an ample divisor $A$ on $X$ so that for any $q$
\begin{equation*}
h^{0}(X,\overline{\mathcal{I}^{q}} \otimes \mathcal{O}_{X}(\lceil mL \rceil + A)) > 0
\end{equation*}
for sufficiently large $m$, where $\overline{\mathcal{I}^{q}}$ denotes the integral closure of $\mathcal{I}^{q}$.
\end{prop}

Note that we can replace $\lceil - \rceil$ by $\lfloor - \rfloor$ by absorbing the difference into $A$.

\begin{proof}
Let $\phi: Y \to X$ denote a smooth resolution of $\mathcal{I}$ and define $E$ by $\mathcal{O}_{Y}(-E) = \phi^{-1}\mathcal{I} \cdot \mathcal{O}_{Y}$.  Suppose that $\varsigma(L,\mathcal{I})=0$ so that $m\phi^{*}L - E$ is not pseudo-effective for any $m$.  Let $p: N^{1}(Y) \to V$ denote the cokernel of the inclusion $\mathbb{R}[\phi^{*}L] \to N^{1}(Y)$.  Note that $p(-E)$ is disjoint from $p(\overline{NE}^{1}(Y))$.  Thus, there is a small ample divisor $H$ on $Y$ so that $p(-E+H)$ is still disjoint from $p(\overline{NE}^{1}(Y))$.  In other words, $m\phi^{*}L - E + H$ is not pseudo-effective for any $m$.

Let $A$ be any ample divisor on $X$.  Choose $q$ so that $qH - \phi^{*}A$ is pseudo-effective.  Then $m\phi^{*}L - qE + \phi^{*}A$ is not pseudo-effective for any $m$.  Thus, for any $A$ there is a $q$ so that
\begin{equation*}
h^{0}(Y,\mathcal{O}_{Y}( \phi^{*}(\lfloor mL \rfloor + A) -qE)) = 0
\end{equation*}
for every $m$.  Since the class of $\lceil mL \rceil - \lfloor mL \rfloor$ is bounded as $m$ varies, by absorbing the difference into $A$ the condition using $\lceil mL \rceil$ also fails.

Conversely, suppose that $\varsigma(L,\mathcal{I}) > 0$.  Then for any real number $b>0$, $a\phi^{*}L - b E$ is pseudo-effective for any $a \geq b / \varsigma(L,\mathcal{I})$.  By Proposition \ref{twistedsectionexistence} (and Remark \ref{floorremark}), there is an ample divisor $H$ on $Y$ (independent of $b$) so that
\begin{equation*}
h^{0}(Y,\mathcal{O}_{Y}(\lfloor c(a\phi^{*}L - bE) \rfloor + H)) > 0
\end{equation*}
for every $c>0$ and every $a \geq b/\varsigma(L,\mathcal{I})$.  Choose an ample $\mathbb{Z}$-divisor $A \geq \phi_{*}H$.  Then $\phi^{*}A \geq \phi^{*}\phi_{*}H \geq H$ so that
\begin{equation*}
h^{0}(Y,\mathcal{O}_{Y}( \phi^{*}(\lceil acL \rceil + A) - \lfloor bcE \rfloor)) > 0.
\end{equation*}
Fix an integer $q$ and choose $c$ so that $\lfloor cbE \rfloor \geq qE$.  Then for any $m > bc/\varsigma(L,\mathcal{I})$ we have
\begin{equation*}
h^{0}(X,\overline{\mathcal{I}^{q}} \otimes \mathcal{O}_{X}(\lceil mL \rceil + A)) > 0.
\end{equation*}
\end{proof}

If we are only interested in whether $\varsigma(L,\mathcal{I})>0$, we can replace the condition of Proposition \ref{nakayamasformulation} by several alternatives.  We have $\mathcal{I}^{q} \subset \overline{\mathcal{I}^{q}} \subset \mathcal{I}^{<q>}$ and by the comparison theorems for symbolic powers (for example \cite{swanson00} Theorem 3.1), there is some $k$ independent of $q$ so that $\mathcal{I}^{<kq>} \subset \mathcal{I}^{q}$.  When $X$ is smooth, we have $\mathcal{I}^{q} \subset \mathcal{J}(\mathcal{I}^{q})$ and by Skoda's theorem $\mathcal{J}(\mathcal{I}^{q}) \subset \mathcal{I}^{q-\dim(X)+1}$ for sufficiently large $q$. Thus, the non-vanishing of $\varsigma(L,\mathcal{I})$ is equivalent to the statement that for any $q$
\begin{equation*}
h^{0}(X,*_{q} \otimes \mathcal{O}_{X}(\lceil mL \rceil + A)) > 0
\end{equation*}
for sufficiently large $m$, where $*_{q}$ can be:
\begin{itemize}
\item $\mathcal{I}^{q}$,
\item $\mathcal{I}^{<q>}$, or
\item $\mathcal{J}(\mathcal{I}^{q})$ when $X$ is smooth.
\end{itemize}

Applying the statement for symbolic powers, we immediately obtain:

\begin{prop} \label{nakayamaconstantbirational}
Let $X$ be a normal variety, $V$ a subvariety not contained in $\Sing(X)$, and $L$ a divisor.  If $(\widetilde{X},\widetilde{V})$ is a smooth $V$-birational model for $(X,V)$, then $\varsigma(\phi^{*}L,\widetilde{V}) > 0$ iff $\varsigma(L,V)>0$.
\end{prop}

The following proposition indicates that the Nakayama constant satisfies the usual compatibility relations.

\begin{prop} \label{nakayamaconstantcompatibility}
Let $X$ be a smooth variety, let $L$ be a pseudo-effective divisor, and let $\mathcal{I}$ be an ideal such that no associated prime of $\mathcal{I}$ is centered in $\mathbf{B}_{-}(L)$.  Then
\begin{enumerate}
\item $\varsigma(L,\mathcal{I}) = \varsigma(P_{\sigma}(L),\mathcal{I})$
\item If $L$ is big, then $\varsigma(L,\mathcal{I}) = \max_{\phi^{*}L \geq A} \varsigma(A,\phi^{-1}\mathcal{I} \cdot \mathcal{O}_{Y})$ where $\phi: Y \to X$ varies over all birational maps and $A$ is big and nef.
\end{enumerate}
\end{prop}

\begin{proof}
\begin{enumerate}
\item  It suffices to show the inequality $\leq$.  Let $\phi: Y \to X$ denote a smooth resolution of $\mathcal{I}$ and let $E$ denote the divisor satisfying $\mathcal{O}_{X}(-E) = \phi^{-1}\mathcal{I} \cdot \mathcal{O}_{Y}$.  Suppose that $\phi^{*}L - \tau E$ is pseudo-effective.  Fix an ample $A$ on $Y$.  For any $\epsilon > 0$, we find that $\phi^{*}L + \epsilon A \sim_{\mathbb{R}} \tau E + F$ for some effective $F$.  Since $\Supp(E)$ is not contained in the diminished base locus of $\phi^{*}L$, we know that $N_{\sigma}(\phi^{*}L + \epsilon A) \leq F$.  Subtracting, we find that $P_{\sigma}(\phi^{*}L + \epsilon A) - \tau E$ is pseudo-effective.  Taking a limit over $\epsilon$ and noting that $\phi^{*}P_{\sigma}(L) \geq P_{\sigma}(\phi^{*}L)$ completes the proof of the inequality.

\item It suffices to show the inequality $\leq$.  We may also replace $L$ by some $\mathbb{Q}$-linearly equivalent divisor so that $L \geq 0$.  Fix an effective ample divisor $H$ on $X$.  Proposition \ref{psigmaapproachesnef} indicates that there are birational maps $\phi_{m}$ and big and nef divisors $N_{m}$ satisfying $N_{m} \leq P_{\sigma}(\phi_{m}^{*}L) \leq N_{m} + \frac{1}{m}\phi_{m}^{*}H$.  The expression on the right hand side can be made arbitrarily close to $\varsigma(P_{\sigma}(L),\phi^{-1}\mathcal{I} \cdot \mathcal{O}_{Y})$.  By (1) this equals $\varsigma(L,\mathcal{I})$.
\end{enumerate}
\end{proof}

\cite{nakayama04} shows that $\varsigma(L,V)$ is controlled by what happens to a very general subvariety of dimension equal to $\dim(V)$.

\begin{prop}[\cite{nakayama04}, Lemma V.2.21] \label{nakayamaconstantgeneralsubvariety}
Let $X$ be a smooth variety of dimension $n$ and let $L$ be a pseudo-effective divisor.  Suppose there is a $d$-dimensional subvariety $V$ such that $\varsigma(L,V) = 0$.  Then there is a very ample divisor $H$ so that any complete intersection $W$ of $(n-d)$ very general elements of $|H|$ satisfies $\varsigma(L,W) = 0$.
\end{prop}

\section{The Numerical Dimension} \label{numericaldimensionsection}

Our goal in this section is to show that the different definitions of the numerical dimension coincide.  We start by giving an example of effective divisors that are numerically equivalent but have different Iitaka dimensions.

\begin{exmple} \label{notnumerical}
We give an example of a threefold $X$ and effective divisors $L,L'$ such that $L \equiv L'$ but $\kappa(L) \neq \kappa(L')$.  Fix an elliptic curve $E$ and consider $S = E \times E$ with projection maps $p_{1}$ and $p_{2}$.  Let $F$ be a fiber of $p_{1}$.  Choose a degree $0$ divisor $T$ on $E$ that is non-torsion and define $N=p_{2}^{*}T$.  We have $\kappa(F) = 1$ and $\kappa(F + N) = -\infty$.

Let $X$ be the $\mathbb{P}^{1}$-bundle $\mathbb{P}_{S}(\mathcal{O}_{S} \oplus \mathcal{O}_{S}(F+N))$ with the morphism $\pi: X \to S$.  Define $L$ to be the section $\mathbb{P}_{S}(\mathcal{O}_{S})$ and define $L' = L - \pi^{*}N$.  Note that $L$ and $L'$ are numerically equivalent.  By identifying the pushforwards of $\mathcal{O}_{X}(mL)$ with symmetric powers of $\mathcal{O}_{S} \oplus \mathcal{O}_{S}(F+N)$ we see that $\kappa(L) = 0$.  Similarly, since $\mathcal{O}_{X}(L')$ can be realized as the relative dualizing sheaf of $\mathbb{P}_{S}(\mathcal{O}_{S}(-N) \oplus \mathcal{O}_{S}(F))$, we see that $\kappa(L') \geq \kappa(F) = 1$.
\end{exmple}

We first prove Theorem \ref{equalityofnumericaldimension} for smooth varieties $X$.  For convenience we arrange the definitions in a more suitable order.  Definition (1) in the following theorem is the definition of numerical dimension in \cite{bdpp04}, while (5) and (6) correspond to $\kappa_{\sigma}(L)$ and $\kappa_{\nu}(L)$ (by Remark \ref{equivalenttonakayama}) in \cite{nakayama04}.  Note that we allow varieties $W \subset \Supp(L)$ at the slight cost of using numerical restrictions in (4).

\begin{thrm} \label{numericaldimensionequalities}
Let $X$ be a smooth variety and let $L$ be a pseudo-effective divisor.  In the following, $A$ will denote some fixed sufficiently ample $\mathbb{Z}$-divisor and $W$ will range over all subvarieties of $X$ not contained in $\mathbf{B}_{-}(L)$.  Then the following quantities coincide:
\begin{enumerate}
\item $\max \{ k \in \mathbb{Z}_{\geq 0} | \langle L^{k} \rangle_{X} \neq 0 \}$.
\item $\max \{ \dim W | \langle L^{\dim W} \rangle_{X|W} > 0 \}$.
\item $\max \{ \dim W | \lim_{\epsilon \to 0} \vol_{X|W}(L + \epsilon A) > 0 \}$.
\item $\max \{ \dim W | \inf_{\phi} \vol_{\widetilde{W}}([P_{\sigma}(\phi^{*}L)]|_{\widetilde{W}}) >0 \}$
where $\phi: (\widetilde{X},\widetilde{W}) \to (X,W)$
ranges over $W$-birational models.
\item $\max \left\{ k \in \mathbb{Z}_{\geq 0} \left| \limsup_{m \to \infty} \frac{h^{0}(X,\lfloor mL \rfloor + A)}
{m^{k}} > 0 \right. \right\}$.
\item $\min \left\{ \dim W | \varsigma(L,W) = 0 \right\}$.

\noindent By convention, if $L$ is big we interpret this expression as returning $\dim(X)$.

\item $\max \{ k \in \mathbb{Z}_{\geq 0} | \,  \exists C>0 \textrm{ such that } Ct^{n-k} < \vol(L+tA) \textrm{ for all } t > 0 \}$.
\end{enumerate}
We call this common quantity the numerical dimension of $L$ and denote it by $\nu_{X}(L)$.  It only depends on the numerical class of $L$.
\end{thrm}

We will prove Theorem \ref{numericaldimensionequalities} using a cycle of inequalities.  The equivalence of (1) - (4) is an easy consequence of the properties of the positive product and the inequality (5) $\leq$ (6) was proved in \cite{nakayama04}, Proposition V.2.22.  The other inequalities will require more work.

\begin{proof}
$(1) = (2)$. Let $H_{1},\ldots,H_{d-k}$ represent very general elements of a very ample linear system.  Since $\langle L^{k} \rangle_{X}$ is in the closure of the cone generated by effective cycles, it is non-zero if and only if $\langle L^{k} \rangle_{X} \cdot H_{1} \cdot \ldots \cdot H_{d-k} > 0$.  By Proposition \ref{restrictedproductandnefdivisors}, this is equivalent to $\langle L^{k} \rangle_{X|H_{1} \cap \ldots \cap H_{d-k}} > 0$.  Thus $(1) \leq (2)$.  By Proposition \ref{ampleintersectiondominates}, the same argument in reverse shows that $(2) \leq (1)$.

\bigskip

$(2) = (3)$.  Proposition \ref{restrictedvolumeandproduct} shows that the conditions set on $W$ in (2) and (3) are the same.

\bigskip

$(3) = (4)$.   Proposition \ref{restrictedvolumeandproduct} allows us to translate between restricted volume and the restricted positive product in the $V$-big case.  Thus, Theorem \ref{positiveproductpsigmaalternative} implies that
\begin{equation*}
\vol_{X|W}(L + \epsilon A) = \inf_{\phi: \widetilde{X} \to X} \vol_{\widetilde{W}}([P_{\sigma}(\phi^{*}(L + \epsilon A))]|_{\widetilde{W}})
\end{equation*}
where $\phi: (\widetilde{X},\widetilde{W}) \to (X,W)$ varies over $W$-birational models.  Consider
\begin{equation*}
\lim_{\epsilon \to 0} \vol_{X|W}(L + \epsilon A) = \lim_{\epsilon \to 0} \inf_{\phi: \widetilde{X} \to X} \vol_{\widetilde{W}}([P_{\sigma}(\phi^{*}(L + \epsilon A))]|_{\widetilde{W}}).
\end{equation*}
Note that on any model $\vol_{\widetilde{W}}([P_{\sigma}(\phi^{*}(L+\epsilon A))]|_{\widetilde{W}})$ is non-decreasing and continuous as a function of $\epsilon$.  Thus, on the right hand side we may commute the limit with the infimum.

\bigskip

$(4) \leq (5)$. The first step is to show that there is some ample divisor on $W$ whose pullback lies beneath each restriction $P_{\sigma}(\phi^{*}L)|_{\widetilde{W}}$.  Using this ample divisor we find a lower bound for the growth of sections of a certain twisted linear series on $W$.  The last step is to prove a lifting theorem for twisted linear series to conclude that $h^{0}(\lfloor mL \rfloor + A)$ satisfies the necessary growth conditions.

\begin{lem} \label{volumederivative}
Let $X$ be a smooth variety of dimension $n$, $L$ a big divisor and $N$ a general element of a big basepoint free linear system.  Then $\vol(L - N) \geq \vol(L) - n \, \vol_{X|N}(L)$.
\end{lem}

The easiest way to demonstrate this is to appeal to the results of \cite{bfj09}.

\begin{proof}
Let $\alpha = \sup_{t \in [0,1]}\{ L - tN \textrm{ is pseudo-effective}\}$.  Note that $1 \geq \alpha$ and since $L$ is big we have $0 < \alpha$.  We will prove the stronger result $\vol(L-N) \geq \vol(L) - n \alpha \, \vol_{X|N}(L)$.

By \cite{bfj09} Corollary C the function $\vol$ is continuously differentiable on the big cone.  More precisely, for $t \in (0,\alpha)$ we have
\begin{equation*}
\frac{d}{dt}\vol(L - tN) = -n \, \vol_{X|N}(L-tN)
\end{equation*}
Note that $\vol_{X|N}(L - tN) \leq \vol_{X|N}(L)$ for any $t \geq 0$.  Thus for every $t \in (0,\alpha)$ there is an inequality
\begin{equation*}
\frac{d}{dt}\vol(L - tN) \geq -n \, \vol_{X|N}(L).
\end{equation*}
Integrating both sides over $t \in [0,\alpha]$, we obtain $\vol(L- \alpha N) \geq \vol(L) - n \alpha \, \vol_{X|N}(L)$.  But if $\alpha \neq 1$, then $\vol(L-\alpha N) = 0 = \vol(L-N)$, finishing the proof.
\end{proof}

\begin{lem} \label{findingbigdivisoronlimit}
Let $W$ be a smooth variety.  Suppose that for every smooth birational model $\phi: \widetilde{W} \to W$ we associate a divisor $L_{\widetilde{W}}$ so that for any birational map $\psi: \widehat{W} \to \widetilde{W}$ we have $\psi^{*}L_{\widetilde{W}} \geq L_{\widehat{W}}$.  Suppose furthermore that
\begin{equation*}
\inf_{\widetilde{W}} \vol(L_{\widetilde{W}}) > 0.
\end{equation*}
There is some ample divisor $H$ on $W$ and some constant $\epsilon$ such that $\vol(L_{\widetilde{W}} - \phi^{*}H) > \epsilon$ for every $\phi$.
\end{lem}

Note that $\vol(L_{\widetilde{W}}) \geq \vol(L_{\widehat{W}})$ for every higher model $\widehat{W}$.

\begin{proof}
For convenience set $n = \dim(W)$ and $\tau = \inf \vol(L_{\widetilde{W}})$.  Fix a very ample divisor $H$ on $W$.  It suffices to show that there is some constant $k$ such that for any smooth model $\phi: \widetilde{W} \to W$ there is an $H' \equiv H$ so that
\begin{equation*}
\vol\left(L_{\widetilde{W}} - \frac{1}{k}\phi^{*}H' \right) > \tau/2.
\end{equation*}
Choose a prime very ample divisor $H' \equiv H$ sufficiently general so that $\psi^{*}H'$ is equal to the strict transform of $H'$.  Note that 
\begin{equation*}
\vol_{\widetilde{W}|\phi^{*}H'}(L_{\widetilde{W}}) \leq \vol_{\widetilde{W}|\phi^{*}H'}(\phi^{*}L_{W})
\end{equation*}
and by \cite{elmnp06} Lemma 2.4 the latter quantity is equal to $\vol_{W|H'}(L_{W})$.  Choose some constant $k$ so that
\begin{equation*}
\frac{1}{k} \vol_{W|H'}(L_{W}) < \frac{\tau}{2n}.
\end{equation*}
(Note that by \cite{bfj09} Proposition 4.8 $k$ is independent of the choice of $H'$, and thus also independent of the choice of $\widetilde{W}$.)  Lemma \ref{volumederivative} implies that
\begin{align*}
\vol(kL_{\widetilde{W}} - \phi^{*}H') & \geq \vol(kL_{\widetilde{W}}) - n \, \vol_{\widetilde{W}|\phi^{*}H'}(kL_{\widetilde{W}}) \\
& \geq \vol(kL_{\widetilde{W}}) - n \, \vol_{\widetilde{W}|\phi^{*}H'}(k\phi^{*}L_{W})
\end{align*}
Rescaling the above expression by $k$ we find
\begin{equation*}
\vol \left(L_{\widetilde{W}} - \frac{1}{k}\phi^{*}H' \right) \geq \vol(L_{\widetilde{W}}) - \frac{n}{k} \, \vol_{\widetilde{W}|\phi^{*}H'}(\phi^{*}L_{W}) > \tau/2
\end{equation*}
proving the claim.
\end{proof}

In our situation, we find:

\begin{cor} \label{havebigdivisor}
Assume that $W$ is a very general intersection of very ample divisors such that $\inf_{\phi} \vol_{\widetilde{W}}(P_{\sigma}(\phi^{*}L)|_{\widetilde{W}}) > 0$ where $\phi: (\widetilde{X},\widetilde{W}) \to (X,W)$ varies over all $W$-birational models.  Then there is an ample divisor $H$ on $W$ so that for any $W$-birational model $\phi: \widetilde{X} \to X$ we have
\begin{equation*}
\vol_{\widetilde{W}}(P_{\sigma}(\phi^{*}L)|_{\widetilde{W}} - \phi^{*}H) > 0.
\end{equation*}
\end{cor}

\begin{proof}
Consider the set of divisors $P_{\sigma}(\phi^{*}L)|_{\widetilde{W}}$.  Since $N_{\sigma}(\phi^{*}L) \geq_{\widetilde{W}} 0$, they satisfy the comparison condition of Lemma \ref{findingbigdivisoronlimit}.  By assumption the inf condition of Lemma \ref{findingbigdivisoronlimit} also holds.  The lemma yields an appropriate ample divisor $H$ on $W$.
\end{proof}

Our next goal is a lifting theorem for twisted linear series.

\begin{prop} \label{restrictedcomparison}
Let $X$ be a smooth variety and let $L$ be an effective divisor.  Suppose that $N$ is a big and nef divisor satisfying $0 \leq N \leq L$ such that $N$ has simple normal crossing support.  Let $|B|$ be a basepoint-free linear system defining a birational morphism on $X$.  For sufficiently general elements $B_{1},\ldots,B_{k} \in |B|$ we have an inequality
\begin{equation*}
h^{0}(W,\mathcal{O}_{W}(K_{W} + \lceil N|_{W} \rceil +  A|_{W}) \leq h^{0}(X|W,\mathcal{O}_{X}(K_{X} + \lceil L \rceil + B_{1} + \ldots + B_{k} + A))
\end{equation*}
where $W$ is the complete intersection $B_{1} \cap \ldots \cap B_{k}$ and $A$ is any nef $\mathbb{Z}$-divisor on $X$.
\end{prop}

\begin{proof}
For convenience define $W_{j} := B_{1} \cap \ldots \cap B_{j}$ and $M_{i} := B_{i+1} + \ldots + B_{k}$.  Note that since the $B_{i}$ are sufficiently general we may assume that each $W_{j}$ is smooth, that $N \geq_{W_{j}} 0$, and that $N|_{W_{j}}$ has simple normal crossing support.  Note furthermore $B$ is big and nef, so that $M_{i}|_{W_{j}}$ is also a big and nef divisor for any $i$ and $j$.

Kawamata-Viehweg vanishing implies that we have surjections
\begin{align*}
H^{0}(W_{i},\mathcal{O}_{W_{i}}&(K_{W_{i}} + \lceil N|_{W_{i}} \rceil + (A + M_{i})|_{W_{i}})) \to \\
& H^{0}(W_{i+1},\mathcal{O}_{W_{i+1}}(K_{W_{i+1}} + \lceil N|_{W_{i}} \rceil|_{W_{i+1}} + (A+M_{i+1})|_{W_{i+1}}))
\end{align*}
Furthermore since $N \geq_{W_{i}} 0$ for every $i$ we have $\lceil N|_{W_{i}} \rceil|_{W_{i+1}} \geq \lceil N|_{W_{i+1}} \rceil$.  Thus by induction we obtain
\begin{align*}
h^{0}(X|W_{i},\mathcal{O}_{X}&(\lceil N \rceil + (K_{X} + A + B_{1} + \ldots + B_{k}))) \\
& \geq h^{0}(W_{i},\mathcal{O}_{W_{i}}(\lceil N|_{W_{i}} \rceil + (K_{X} + A+ B_{1} + \ldots + B_{k})|_{W_{i}})).
\end{align*}
When $i=k$, we obtain the desired statement.
\end{proof}

We now finish the proof of the inequality (4) $\leq$ (5).   Set $k$ to be the value of (4).  Fix an ample divisor $A$ on $X$ as in Theorem \ref{restrictedbaselocuscountable} so that for any $m$ there is an $L_{m} \sim \lceil mL \rceil + A$ such that $L_{m} \geq 0$.

For each $L_{m}$, we can apply Proposition \ref{psigmaapproachesnef} to find an effective divisor $G_{m}$, a countable sequence of maps $\phi_{i,m}$, and a big and nef divisor $N_{i,m}$ satisfying
\begin{equation*}
N_{i,m} \leq P_{\sigma}(\phi_{i,m}^{*}L_{m}) \leq N_{i,m} +\frac{1}{i} \phi_{i,m}^{*}G_{m}.
\end{equation*}
We may of course assume that each $N_{i,m}$ has simple normal crossing support and each $\phi_{i,n}$ is a composition of blow-ups along smooth centers.

Note that the set of maps $\phi_{i,m}$ is countable as $m$ and $i$ vary.  Fix a very ample linear system $|B|$ on $X$.  We can choose very general elements $B_{1},\ldots,B_{k} \in |B|$ so that the $\phi_{i,m}^{*}B_{j}$ satisfy the conditions of Proposition \ref{restrictedcomparison} for each $\widetilde{X}_{i,m}$ and $N_{i,m}$ simultaneously.  We may also choose the $B_{j}$ sufficiently general so that the strict transform of $B_{j}$ over $\phi_{i,m}$ is the same as the pullback for every $i$ and $m$.  Set $W = B_{1} \cap \cdots \cap B_{k}$.   Then each $\phi_{i,m}$ is $W$-birational and $\widetilde{W}_{i,m,j} = \phi_{i,m}^{*}B_{1} \cap \ldots \cap \phi_{i,m}^{*}B_{j}$ is smooth for every $j$ between $1$ and $k$.

Choose an ample divisor $H$ on $W$ as in Corollary \ref{havebigdivisor}.  For each $G_{m}$, choose a sufficiently small $\epsilon_{m}>0$ so that $H-\epsilon_{m}G_{m}|_{W}$ is pseudo-effective.  By choosing $i > 1/\epsilon_{m}$, we find models $\phi_{m}: \widetilde{X}_{m} \to X$ so that
\begin{equation*}
N_{m} \leq_{\widetilde{W}_{m}} P_{\sigma}(\phi_{m}^{*}L_{m}) \leq_{\widetilde{W}_{m}} N_{m} + \epsilon_{m} \phi_{m}^{*}G_{m}.
\end{equation*}
Thus
\begin{align*}
N_{m}|_{\widetilde{W}_{m}} - (m-1)\phi_{m}^{*}H 
& \geq (P_{\sigma}(\phi_{m}^{*}L_{m})  - \epsilon_{m}\phi_{m}^{*}G_{m})|_{\widetilde{W}_{m}} - (m - 1) \phi_{m}^{*}H \\
& \geq (P_{\sigma}(\phi_{m}^{*}L_{m})  - P_{\sigma}(\phi_{m}^{*}mL))|_{\widetilde{W}_{m}} \\
& \qquad + m(P_{\sigma}(\phi_{m}^{*}L)|_{\widetilde{W}_{m}} -  \phi_{m}^{*}H) \\
& \qquad +  \phi_{m}^{*}(H - \epsilon_{m}G_{m}|_{W}).
\end{align*}
We analyze this last sum term by term.  Since $L_{m}-mL$ is $W$-pseudo-effective and $N_{\sigma}(\phi_{m}^{*}L) \geq_{\widetilde{W}_{m}} 0$, the first term is pseudo-effective by Proposition \ref{vbignessandpsigma}.  The conclusion of Corollary \ref{havebigdivisor} is that the second term is big.  The third term is also pseudo-effective by construction.  Thus $D_{m} := N_{m}|_{\widetilde{W}_{m}} - (m-1)\phi_{m}^{*}H$ is big. 

Fix a very ample divisor $M$ on $X$.  Then
\begin{align*}
h^{0}&(\widetilde{W_{m}},\mathcal{O}_{\widetilde{W_{m}}} (K_{\widetilde{W_{m}}} + (k+2)\phi_{m}^{*}M|_{W} + \lceil N_{m}|_{\widetilde{W}_{m}} \rceil)) \\ & \geq h^{0}(\widetilde{W_{m}},\mathcal{O}_{\widetilde{W_{m}}}(K_{\widetilde{W_{m}}} + (k+2)\phi_{m}^{*}M|_{W} + \lceil D_{m} \rceil + \lfloor (m-1)\phi_{m}^{*}H \rfloor)) \\
& \geq h^{0}(W,\lfloor (m-1)H \rfloor) \textrm{ by Proposition \ref{twistedsectionexistence}} \\
& \geq Cm^{k}
\end{align*}
for some constant $C>0$ and for $m$ sufficiently large.

We conclude by applying Proposition \ref{restrictedcomparison}.  We have already chosen the divisors $B_{1},B_{2},\ldots,B_{k}$ sufficiently general so that their pullbacks satisfy the conditions of the theorem.  
For convenience define $A' = B_{1} + \ldots + B_{k}$.  Lemma \ref{restrictedcomparison} shows that the dimensions of the spaces of restricted sections
\begin{equation*}
h^{0}(\widetilde{X}_{m}|\widetilde{W}_{m},\mathcal{O}_{\widetilde{X}_{m}}(K_{\widetilde{X}_{m}} + \phi_{m}^{*}(L_{m} + A' + (k+2)\phi^{*}M))) > Cm^{k}
\end{equation*}
for some constant $C>0$ and for sufficiently large $m$.  Since $K_{\widetilde{X}_{m}/X}$ is $\phi_{m}$-exceptional, these dimensions are equal to
\begin{align*}
h^{0}(X|W,\mathcal{O}_{X}&(K_{X} + L_{m} + A' + (k+2)\phi^{*}M)) \\
& = h^{0}(X|W,\mathcal{O}_{X}(K_{X} + \lceil mL \rceil + A + A' + (k+2)\phi^{*}M))
\end{align*} 
Thus $h^{0}(X,\mathcal{O}_{X}(K_{X} + \lceil mL \rceil + A + A' + (k+2)\phi^{*}M))$ is also bounded below by $Cm^{k}$ for sufficiently large $m$.

\bigskip

$(5) \leq (6)$.  This is proved in \cite{nakayama04} Proposition V.2.22.

\bigskip

$(6) \leq (1)$.  By Proposition \ref{nakayamaconstantgeneralsubvariety}, we may assume that $W$ is a very general intersection of very ample divisors.  We need to consider the $0$-case separately.  Note that (1) is $0$ precisely when $P_{\sigma}(L)$ is numerically trivial.  This means that (6) is also $0$.  Thus, we can prove that $(6) \leq (1)$ by considering the case where (6) is at least $2$ and (1) is at least $1$.

\bigskip

Suppose for a contradiction that (1) is less than the value of (6).  For convenience we set $k$ to be the value of (1).  Let $W$ be a $k$-dimensional intersection of very general very ample divisors.  Set $\tau = \varsigma(L,W) > 0$, and let $\phi: Y \to X$ be the blow-up of $W$ with exceptional divisor $E$.

Fix a very ample divisor $H$ on $Y$.  We begin by analyzing $\phi^{*}L + \epsilon H$.  Choose models $\psi_{i}: \widetilde{Y}_{i} \to Y$ computing the positive products $\langle (\phi^{*}L + \epsilon H)^{k} \rangle_{Y|E}$ and $\langle (\phi^{*}L + \epsilon H)^{k+1} \rangle_{Y}$.  Choose big and nef divisors $A_{i} \leq \psi_{i}^{*}(\phi^{*}L + \epsilon H)$ on $\widetilde{Y}_{i}$ that compute the product.  By Proposition \ref{nakayamaconstantcompatibility}, $P_{\sigma}(\psi_{i}^{*}(\phi^{*}L + \epsilon H)) - \tau \psi_{i}^{*}E$ is always pseudo-effective, so by choosing $\psi_{i}$ appropriately we may also assume that $A_{i} - \frac{\tau}{2}\psi_{i}^{*}E$ is pseudo-effective for each $A_{i}$.  Thus $A_{i} - \frac{\tau}{2}\widetilde{E}$ is also pseudo-effective, where $\widetilde{E}$ denotes the strict transform of $E$ on $\widetilde{Y}_{i}$.   Then
\begin{equation*}
0  \leq \left( A_{i} - \frac{\tau}{2} \widetilde{E} \right) \cdot A_{i}^{k} \cdot \psi_{i}^{*}H^{d-k-1}.
\end{equation*}
By taking a limit over pushforwards on all such models, we find
\begin{equation*}
0 \leq \langle (\phi^{*}L + \epsilon H)^{k+1} \rangle_{Y} \cdot H^{d-k-1} - \frac{\tau}{2} \langle (\phi^{*}L + \epsilon H)^{k} \rangle_{Y|E} \cdot H^{d-k-1}.
\end{equation*}
This is true for all sufficiently small $\epsilon$, so
\begin{equation*}
0 \leq \langle \phi^{*}L^{k+1} \rangle_{Y} \cdot H^{d-k-1} - \frac{\tau}{2} \langle \phi^{*}L^{k} \rangle_{Y|E} \cdot H^{d-k-1}.
\end{equation*}
By choosing sufficiently general elements $H_{1},\ldots,H_{d-k-1} \in |H|$, we may ensure that $E \cap H_{1} \cap \ldots \cap H_{d-k-1}$ maps finitely onto $W$ via $\phi$.  Letting the $A_{1},\ldots,A_{d-k}$ denote the ample divisors whose intersection is $W$, we have
\begin{align*}
\langle \phi^{*}L^{k} \rangle_{Y|E} \cdot H^{d-k-1} & = \langle \phi^{*}L^{k} \rangle_{Y|E \cap H_{1} \cap \ldots \cap H_{d-k-1}} \\
& = C \langle L^{k} \rangle_{X|W} \\
& = C \langle L^{k} \rangle_{X} \cdot A_{1} \cdot \ldots \cdot A_{d-k}
\end{align*}
for some positive constant $C$.  By assumption this latter quantity is positive, so
\begin{equation*}
0 < \langle \phi^{*}L^{k+1} \rangle_{Y} \cdot H^{d-k-1}
\end{equation*}
contradicting the fact that $\langle L^{k+1} \rangle_{X} = 0$.

\bigskip

$(7) \leq (1)$.  Let $k$ denote the value of (1).  Note that
\begin{align*}
t^{n-k} \langle (L+tA)^{k}  \rangle \cdot A^{n-k} & = \langle (L+tA)^{k} \cdot (tA)^{n-k} \rangle \\
& \leq \langle (L + tA)^{n} \rangle.
\end{align*}
(1) implies that there is some constant $C$ such that $C < \langle (L+tA)^k \rangle \cdot A^{n-k}$ for every $t>0$.  Thus we obtain $Ct^{n-k} < \vol(L+tA)$ for every $t>0$.

\bigskip

$(1) \leq (7)$.  Let $k$ denote the value of (7).  For every constant $C$ there is some $t>0$ such that
\begin{equation*}
\langle (L+tA)^{n} \rangle < Ct^{n-k-1}.
\end{equation*}
This implies that
\begin{align*}
t^{n-k-1} \langle (L+tA)^{k+1}  \rangle \cdot A^{n-k-1} < Ct^{n-k-1}
\end{align*}
so that for any $C$ there is some $t$ such that $\langle (L+tA)^{k+1} \rangle \cdot A^{n-k-1} < C$.  Note that the left hand side is increasing in $t$, so that the inequality must hold for arbitrarily small $t$.  Thus the value of (1) is at most $k$.
\end{proof}

The numerical dimension satisfies a number of natural properties.  All of the following are checked in \cite{nakayama04} Proposition V.2.7 except for (\ref{psigmaprop}) and (\ref{growthprop}).

\begin{thrm}[\cite{nakayama04}, Proposition V.2.7] \label{numdimproperties}
Let $X$ be a smooth variety and let $L$ be a pseudo-effective $\mathbb{R}$-divisor.
\begin{enumerate}
\item We have $0 \leq \nu(L) \leq \dim(X)$ and $\kappa(L) \leq \nu(L)$.
\item $\nu(L) = \dim(X)$ iff $L$ is big and $\nu(L) = 0$ iff $P_{\sigma}(L) \equiv 0$.
\item If $L'$ is pseudo-effective, then $\nu(L + L') \geq \nu(L)$.
\item If $f: Y \to X$ is any surjective morphism from a normal variety $Y$ then
$\nu(f^{*}L) = \nu(L)$.
\item \label{psigmaprop} $\nu(L) = \nu(P_{\sigma}(L))$.
\item Suppose that $f: X \to Z$ has connected fibers and $F$ is a very general fiber of $f$.  Then
$\nu(L) \leq \nu(L|_{F}) + \dim(Z)$.
\item \label{growthprop} Fix some sufficiently ample $\mathbb{Z}$-divisor $A$.  Then there are positive constants $C_{1},C_{2}$ so that
\begin{equation*}
C_{1}m^{\nu(L)} < h^{0}(X,\mathcal{O}_{X}(\lfloor mL \rfloor + A)) < C_{2}m^{\nu(L)}
\end{equation*}
for every sufficiently large $m$.
\end{enumerate}
\end{thrm}

\begin{proof}
(\ref{psigmaprop}) follows from the invariance of the positive product under passing to $P_{\sigma}$.

Consider the inequality of (\ref{growthprop}).  The leftmost inequality was stated explicitly while demonstrating the implication $(4) \leq (5)$ in the proof of Theorem \ref{numericaldimensionequalities}.  To show the rightmost inequality, let $W$ be a subvariety of dimension $\nu(L)$ with $\varsigma(L,W)=0$.  Proposition \ref{nakayamasformulation} (and the following discussion) shows that there is a positive integer $q$ with
\begin{equation*}
h^{0}(X,\mathcal{I}_{W}^{q} \otimes \mathcal{O}_{X}(\lceil mL \rceil + A)) = 0
\end{equation*}
for sufficiently large $m$.  Writing $W_{q}$ for the subscheme defined by the ideal $\mathcal{I}_{W}^{q}$, for sufficiently large $m$ there is an injection
\begin{equation*}
h^{0}(X,\mathcal{O}_{X}(\lceil mL \rceil + A)) \to h^{0}(W_{q},\mathcal{O}_{W_{q}}(\lceil mL \rceil + A))
\end{equation*}
and the rate of growth of the latter is bounded by $m^{\dim(W_{q})} = m^{\nu(L)}$.
\end{proof}

It is interesting to note that $\nu$ is \emph{not} lower semicontinuous as might be expected.  This is a consequence of the fact that the restricted positive product is only semi-continuous on the boundary of the $V$-pseudo-effective cone.

\begin{exmple}[\cite{bfj09}, Example 3.8] \label{notlowersemicontinuous}
Let $X$ be any smooth surface with infinitely many $-1$ curves.  Take some compact slice of $\overline{NE}^{1}(X)$.  We can choose a convergent sequence of distinct classes $\{ \alpha_{i} \}$ on this compact slice such that each $\alpha_{i}$ lies on a ray generated by a different $-1$ curve.  Note that for any irreducible curve $C$ there is at most one $i$ for which $\alpha_{i} \cdot C < 0$.  Thus $\beta := \lim_{i \to \infty} \alpha_{i}$ must be a nef class.  A non-trivial nef class $\beta$ has $\nu(\beta) \geq 1$, but $\nu(\alpha_{i}) = 0$ for every $i$.  Thus $\nu$ is not lower semi-continuous.
\end{exmple}

\begin{ques}
What properties does $\nu$ satisfy along the $V$-pseudo-effective boundary?
\end{ques}

\subsection{The Numerical Dimension for Normal Varieties}

Since the numerical dimension is a birational invariant, we can extend the definition to any normal variety $X$.

\begin{defn}
Let $X$ be a normal variety and let $L$ be an $\mathbb{R}$-Cartier divisor on $X$.  We define $\nu(L)$ to be $\nu(f^{*}L)$ where $f: Y \to X$ is any smooth model.
\end{defn}

We now complete the proof of Theorem \ref{equalityofnumericaldimension} by showing that the criteria of Theorem \ref{numericaldimensionequalities} can be applied directly to a normal variety.  Note that the numbering in the two theorems is different; we will use the numbering of Theorem \ref{equalityofnumericaldimension}.

\begin{proof}[Proof of Theorem \ref{equalityofnumericaldimension}:]
$(1) = \nu(L)$ since the arguments in the proof of \cite{nakayama04} Proposition V.2.7 show that (1) is a birational invariant even for normal varieties.

We next show that $(3) = \nu(L)$.  We first claim that there is a complete intersection $W$ of very general very ample divisors that maximizes (3).  Suppose that $V \subset X$ is a $k$-dimensional subvariety that achieves the maximum value in (3).  Choose very ample divisors $A_{1}, \ldots, A_{n-k}$ whose (scheme-theoretic) complete intersection $W_{0}$ contains $V$ and also has dimension $k$.  Set $P = \mathbb{P}H^{0}(X,\mathcal{O}_{X}(A_{1})) \times \ldots \times \mathbb{P}H^{0}(X,\mathcal{O}_{X}(A_{n-k}))$.

Let $\mathcal{J}$ be the ideal sheaf on $X \times P$ whose restriction to a fiber of the second projection is the ideal sheaf of the corresponding complete intersection on $X$.  Note that $\mathcal{J}$ is flat over the locus on $P$ representing intersections of the expected dimension.  By upper-semicontinuity, we find that for any fixed divisor $D$ we have
\begin{equation*}
h^{0}(X,\mathcal{I}_{W}(\lfloor D \rfloor)) \leq h^{0}(X,\mathcal{I}_{W_{0}}(\lfloor D \rfloor))
\end{equation*}
for a general complete intersection $W$.  Thus
\begin{align*}
h^{0}(X|W,\mathcal{O}_{X}(\lfloor D \rfloor)) & \geq h^{0}(X|W_{0},\mathcal{O}_{X}(\lfloor D \rfloor)) \\
& \geq h^{0}(X|V,\mathcal{O}_{X}(\lfloor D \rfloor))
\end{align*}
since the restriction map $\mathcal{O}_{X} \to \mathcal{O}_{V}$ factors through restriction to $\mathcal{O}_{W_{0}}$.  In particular, if we fix a countable collection of divisors $D_{i}$, then for a very general complete intersection $W$ we have $\vol_{X|W}(D_{i}) \geq \vol_{X|V}(D_{i})$ for every $i$.  Setting $D_{i} := L + \frac{1}{i}A$ yields the claim.

Let $\phi: Y \to X$ be a smooth model of $X$.  For any ample divisor $A$ on $Y$, there is an ample divisor $H$ on $X$ such that $\phi^{*}H \geq A$.  Since $W$ is not contained in any $\phi$-exceptional center, we may furthermore ensure that $\Supp(\phi^{*}H - A)$ does not contain $W$.

In particular, for any ample divisor $A$ on $Y$ there is some $H$ on $X$ such that
\begin{equation*}
\vol_{X|W}(L+\epsilon H) = \vol_{Y|\widetilde{W}}(\phi^{*}(L + \epsilon H)) \geq \vol_{Y|\widetilde{W}}(\phi^{*}L + \epsilon A).
\end{equation*}
Similarly, for any ample divisor $H$ on $X$ there is an $A$ on $Y$ with $A - \phi^{*}H$ ample.  Thus  $(3) = \nu(L)$ is proved.

$(2) = \nu(L)$ follows from the arguments of the previous two paragraphs.

$(4) = \nu(L)$ since (4) remains unchanged upon passing to a smooth $V$-birational model.

Both $(5) = \nu(L)$ and $(6) = \nu(L)$ follow from Corollary \ref{nonvanishingposprod}.

$(7) = \nu(L)$ by Proposition \ref{nakayamaconstantbirational}.

\end{proof}

\section{The Restricted Numerical Dimension} \label{restrictednumericaldimensionsection}

We now turn to the restricted numerical dimension.  For a subvariety $V$, $\nu_{X|V}(L)$ should measure the maximal dimension of a subvariety $W \subset V$ such that the ``positive restriction'' of $L$ to $W$ is big.

\begin{thrm} \label{restrictednumericaldimensionequalities}
Let $X$ be a smooth variety, let $V$ be a subvariety of $X$, and let $L$ be a $V$-pseudo-effective divisor.  In the following, $A$ will denote some fixed sufficiently ample $\mathbb{Z}$-divisor and $W$ will range over all subvarieties of $V$ not contained in $\mathbf{B}_{-}(L)$.  Then the following quantities coincide:
\begin{enumerate}
\item $\max \{ k \in \mathbb{Z}_{\geq 0} | \langle L^{k} \rangle_{X|V} \neq 0 \}$.
\item $\max \{ \dim W | \langle L^{\dim W} \rangle_{X|W} > 0 \}$.
\item $\max \{ \dim W | \lim_{\epsilon \to 0} \vol_{X|W}(L + \epsilon A) > 0 \}$.
\item $\max \{ \dim W | \liminf_{\phi} \vol_{\widetilde{W}}([P_{\sigma}(\phi^{*}L)]|_{\widetilde{W}}) >0 \}$ where $\phi: (\widetilde{X},\widetilde{W}) \to (X,W)$
ranges over $W$-birational models.
\end{enumerate}
This common quantity is known as the restricted numerical dimension of $L$ along $V$ and is denoted $\nu_{X|V}(L)$.
It only depends on the numerical class of $L$.
\end{thrm}

The argument is the same as in the proof of the first four equivalences in Theorem \ref{numericaldimensionequalities}.  One wonders whether the other equalities in Theorem \ref{numericaldimensionequalities} can be extended to analogous notions for the restricted numerical dimension.  Perhaps the most important is the restricted version of $\kappa_{\sigma}$:

\begin{defn}
Let $X$ be a smooth variety, $V$ a subvariety, and $L$ a $V$-pseudo-effective divisor.  Fix any divisor $A$.  If $H^{0}(X|V,\mathcal{O}_{X}(\lfloor mL + A \rfloor))$ is non-zero for infinitely many values of $m$, we define
\begin{equation*}
\kappa_{\sigma}(X|V,L;A) := \max \left\{ k \in \mathbb{Z}_{\geq 0} \left| \limsup_{m \to \infty} \frac{h^{0}(X|V, \mathcal{O}_{X}(\lfloor mL + A \rfloor))}
{m^{k}} > 0 \right. \right\}.
\end{equation*}
Otherwise, we define $\kappa_{\sigma}(X|V,L;A) = -\infty$.
The restricted $\sigma$-dimension $\kappa_{\sigma}(X|V,L)$ is defined to be
\begin{equation*}
\kappa_{\sigma}(X|V,L) := \max_{A} \{ \kappa_{\sigma}(X|V,L;A) \}.
\end{equation*}
\end{defn}

Arguing as in the proof of \cite{nakayama04} Proposition V.2.7, one can check that the restricted $\sigma$-dimension is a numerical and birational invariant.

\begin{ques}
Let $X$ be a smooth variety, $V$ be a subvariety, and $L$ be a $V$-pseudo-effective divisor.  Does $\nu_{X|V}(L) = \kappa_{\sigma}(X|V,L)$?
\end{ques}

Since the restricted numerical dimension is invariant under passing to admissible models, we can extend the definition to pairs with singularities.

\begin{defn}
Let $X$ be a normal variety, $V$ a subvariety not contained in $\Sing(X)$, and $L$ a $V$-pseudo-effective divisor.  We define $\nu_{X|V}(L) = \nu_{Y|W}(f^{*}L)$ where $(Y,W)$ is any smooth $V$-birational model of $(X,V)$.
\end{defn}

\subsection{Properties of the Restricted Numerical Dimension}

The restricted numerical dimension satisfies similar properties to the numerical dimension.  Since we know less about $\nu_{X|V}$, the statements are slightly weaker.

\begin{thrm} \label{restrictednumericaldimensionproperties}
Let $X$ be a smooth variety, $V$ a subvariety of $X$, and $L$ a $V$-pseudo-effective divisor.
\begin{enumerate}
\item $\nu_{X|V}(L) \leq \nu(L)$ and if $V$ is normal then $\nu_{X|V}(L) \leq \nu(L|_{V})$.
\item $\nu_{X|V}(L) = \nu_{X|V}(P_{\sigma}(L))$.
\item When $L$ is nef, $\nu_{X|V}(L) = \nu_{V}(L|_{V})$.
\item If $L'$ is also $V$-pseudo-effective, then $\nu_{X|V}(L + L') \geq \nu_{X|V}(L)$.
\item Suppose that $\nu_{X|V}(L) < \dim(V)$.  If $H$ is a very general very ample divisor on $X$ then $\nu_{X|V}(L) = \nu_{X|V \cap H}(L)$
\item If $\phi: (\widetilde{X},\widetilde{V}) \to (X,V)$ is an admissible model with $\widetilde{X}$ smooth then $\nu_{\widetilde{X}|\widetilde{V}}(\phi^{*}L) = \nu_{X|V}(L)$.
\item Let $\phi: Y \to X$ be a smooth birational model and let $W$ be a subvariety of $Y$ such that $\phi|_{W}$ maps surjectively onto $V$.  Then $\nu_{Y|W}(\phi^{*}L) = \nu_{X|V}(L)$.
\end{enumerate}
\end{thrm}

\begin{proof} $ $

\begin{enumerate}
\item Note that if $Z$ and $Z'$ are subvarieties of $X$ with $Z \subset Z'$, then $\vol_{X|Z'}(L) \geq \vol_{X|Z}(L)$ since the restriction map on sections of $L$ from $X$ to $Z$ factors through the restriction map to $Z'$.

Fix an ample divisor $A$ on $X$ and let $W$ be an intersection of very general very ample divisors on $X$.  The two inequalities follow from the fact that $\vol_{X|W}(L + \epsilon A) \geq \vol_{X|V \cap W}(L + \epsilon A)$ and $\vol_{X|V}(L + \epsilon A) \geq \vol_{X|V \cap W}(L + \epsilon A)$.  

\item This follows from the fact that the restricted positive product is invariant under passing to $P_{\sigma}$ as demonstrated in Proposition \ref{positiveproductpositivepart}.

\item The restricted volume of an ample divisor can be calculated as an intersection product, so the equality follows from characterization (3) in Theorem \ref{restrictednumericaldimensionequalities}.

\item Fix an ample divisor $A$.  The inequality follows from $\vol_{X|W}(L+L'+2\epsilon A) \geq \vol_{X|W}(L+\epsilon A)$.

\item Using characterization (1) in Theorem \ref{restrictednumericaldimensionequalities}, we see that if $k < \dim V$ then $\langle L^{k} \rangle_{X|V} \neq 0$ iff $\langle L^{k} \rangle_{X|V \cap H} = \langle L^{k} \rangle_{X|V} \cdot H \neq 0$.

\item This is a consequence of Proposition \ref{restrictedpositiveproductadmissiblemodels} which describes how the restricted positive product is compatible with admissible models.

\item First suppose that $\dim(W) > \dim(V)$; we show that $\nu_{Y|W}(\phi^{*}L) < \dim(W)$.  Every fiber of $\phi|_{W}$ is covered by curves satisfying $\phi^{*}L \cdot C = 0$.  Since $\mathbf{B}_{-}(\phi^{*}L) = \phi^{-1}\mathbf{B}_{-}(L)$, the general such curve avoids $\mathbf{B}_{-}(\phi^{*}L)$.  In particular, for any $W$-birational model $\psi: \widetilde{Y} \to Y$, the subvariety $\widetilde{W}$ is covered by curves satisfying $P_{\sigma}(\psi^{*}\phi^{*}L) \cdot C = 0$.  Thus $\nu_{Y|W}(\phi^{*}L) < \dim(W)$ by characterization (4) in Theorem \ref{restrictednumericaldimensionequalities}.

Fix a very general very ample divisor $H$ on $Y$.  By property (5), $\nu_{Y|W}(\phi^{*}L) = \nu_{Y|W \cap H}(\phi^{*}L)$.  Proceeding inductively, we reduce to the case when $\dim(W) = \dim(V)$, which is (6).
\end{enumerate}
\end{proof}

It is important to note that we can have $\nu_{X|V}(L) = \dim(V)$ even when $L$ is not $V$-big.

\begin{exmple}
Let $X$ be a smooth variety, $V$ a smooth subvariety, and $L$ a $V$-big divisor.  Let $\phi: (Y,W) \to (X,V)$ be an admissible model such that some $\phi$-exceptional center contains $V$.  Then $\phi^{*}L$ is $W$-pseudo-effective but not $W$-big. Nevertheless, the invariance of $\nu_{X|V}(L)$ under passing to admissible models shows that we still have $\nu_{X|V}(L) = \dim(V)$.
\end{exmple}

We next show that the non-vanishing of $\nu(L)$ can be detected by the restricted numerical dimension $\nu_{X|C}(L)$ for a very general curve $C$.

\begin{prop}
Let $X$ be a smooth variety and let $L$ be a pseudo-effective divisor on $X$.  Then $\nu(L) > 0$ iff there is a curve $C$ on $X$ defined as a very general complete intersection of very ample divisors with $\nu_{X|C}(L) > 0$.
\end{prop}

\begin{proof}
If $\nu(L) = 0$ then $\nu_{X|C}(L) = 0$ by Theorem \ref{restrictednumericaldimensionproperties}.

Conversely, suppose that $C$ is a very general intersection of very ample divisors.  By choosing $C$ appropriately, we may assume that it avoids every component of $\mathbf{B}_{-}(P_{\sigma}(L))$.  In particular, for any $C$-birational model $\phi: Y \to X$ we have
\begin{equation*}
\vol(P_{\sigma}(\phi^{*}L)|_{\widetilde{C}}) = P_{\sigma}(\phi^{*}L) \cdot \widetilde{C} = \phi^{*}P_{\sigma}(L) \cdot \widetilde{C} = P_{\sigma}(L) \cdot C.
\end{equation*}
Thus, if $\nu_{X|C}(L)=0$, then $P_{\sigma}(L) \cdot C = 0$.  But since $C$ is an intersection of ample divisors, this implies that $P_{\sigma}(L) \equiv 0$ and $\nu(L) = 0$.
\end{proof}

\nocite{*}
\bibliographystyle{amsalpha}
\bibliography{numdim}

\end{document}